\documentclass[]{amsart}

\makeatletter
\@addtoreset{equation}{section}

\makeatother

\usepackage{braket} 
\usepackage{amssymb} 
\usepackage{color}
\usepackage[pdftex]{graphicx}
\usepackage{mathrsfs} 
\usepackage{url}

\newtheorem{theorem}{Theorem}[section]
\newtheorem{proposition}[theorem]{Proposition}

\newtheorem{lemma}[theorem]{Lemma}

\theoremstyle{definition}
\newtheorem{definition}[theorem]{Definition}
\newtheorem{remark}[theorem]{Remark}
\newtheorem{assumption}[theorem]{Assumption}

\newcommand{\R}{\mathbb{R}}

\newcommand{\C}{\mathbb{C}}

\newcommand{\N}{\mathbb{N}}
\newcommand{\vk}{\kappa}
\newcommand{\vp}{\varphi}

\newcommand{\va}{\alpha}
\newcommand{\vb}{\beta}
\newcommand{\vc}{\gamma}

\newcommand{\ve}{\varepsilon}
\newcommand{\vz}{\zeta}

\newcommand{\vt}{\tau}

\newcommand{\pd}{\partial}

\newcommand{\K}{\mathcal{K}}
\newcommand{\W}{\mathcal{W}}

\newcommand{\wto}{\rightharpoonup}

\newcommand{\dx}{\,dx}
\newcommand{\dy}{\,dy}
\newcommand{\dt}{\,dt}
\newcommand{\ds}{\,ds}
\newcommand{\dvt}{\,d\tau}
\newcommand{\SC}{\mathsf{SC}}
\newcommand{\Msym}{M_{\rm sym}}
\newcommand{\half}{\Big(\frac{1}{2}\Big)}

\begin{document}

\title[elastic graphs with the symmetric cone obstacle]{A remark on elastic graphs with the symmetric cone obstacle}

\author{Kensuke Yoshizawa}
\address[K.~Yoshizawa]{Institute of Mathematics for Industry, Kyushu University, 744 Motooka, Nishi-ku, Fukuoka 819-0395, Japan}
\email{k-yoshizawa@imi.kyushu-u.ac.jp}

\keywords{obstacle problem; elastic energy; shooting method; fourth order.}
\subjclass[2020]{49J40; 34B15; 53A04; 53C44}

\date{\today}

\begin{abstract}
This paper is concerned with the variational problem for the elastic energy defined on symmetric graphs under the unilateral constraint.
Assuming that the obstacle function satisfies the symmetric cone condition, we prove (i) uniqueness of minimizers, (ii) loss of regularity of minimizers, and give (iii) complete classification of  existence and non-existence of  minimizers in terms of the size of obstacle. 
As an application, we characterize the solution of the obstacle problem as equilibrium of the corresponding dynamical problem.
\end{abstract}

\maketitle

\section{Introduction} \label{Ssection:1}

For a given smooth curve $\vc\subset\R^2$, Bernoulli--Euler's elastic energy, also known as bending energy, is defined by 
\[\W(\vc)=\int_{\vc}\vk^2\ds,\]
where $\vk$ and $s$ denote the curvature and the arclength parameter of $\vc$, respectively. 
The variational problems of $\W$ have been studied as the model of the elastic rods and due to the geometric interest.
In this paper we consider curves given as graphs with fixed ends.
For a curve written as the graph of a function $u:[0,1]\to\R$, the elastic energy of the curve $(x, u(x))$ is given by 
$$
\W(u) := \int_0^1 \vk_u(x)^2 \sqrt{1+u'(x)^2} \dx = \int_0^1 \left(\frac{u''(x)}{\left(1+u'(x)^2\right)^{\frac{3}{2}}}\right)^2 \sqrt{1+u'(x)^2} \dx,
$$
where $\vk_u$ denotes the curvature of the curve $(x,u(x))$.

Recently, the obstacle problems for $\W$ were studied by \cite{DD, miura_16, Muller01}.
This paper is concerned with the minimization problem for $\W$ with the unilateral constraint that the curve lies above a given function $\psi:[0,1]\to\R$. 
That is, we consider  
\begin{align}\label{Seq:M}\tag{M}
\min_{v\in M_{\rm sym}} \W(v), 
\end{align}
where $\Msym$ is a convex set of $H(0,1):=H^2(0,1)\cap H^1_0(0,1)$ as follows:
\[
M_{\rm sym}:=\Set{v\in H(0,1) | v\geq \psi \ \ \text{in} \ \ [0,1], \quad v(x)=v(1-x) \ \ \text{for} \ \ 0\leq x \leq \frac{1}{2}}.
\]
In this paper we say that $u$ is a solution of \eqref{Seq:M} if $u\in\Msym$ attains $\inf_{v\in\Msym}\W(v)$. 
We shall assume the following:
\begin{assumption}\label{Sassumption:1.1}
We say that $\psi:[0,1]\to\R$ satisfies the \textit{symmetric cone condition} if the following hold:
\begin{itemize} 
\item[(i)] $\psi(x)=\psi(1-x)$ \quad for \quad $x\in[0,1]$;
\item[(ii)] $\psi(0)<0$, \ $\psi(\tfrac{1}{2})>0$ \ \ and 
\end{itemize}
\begin{align} \label{S0726-1}
\psi(x)=(1-2x ) \psi(0) +2x\,\psi\Big(\frac{1}{2}\Big) \quad \text{for}\quad 
0\leq x\leq \frac{1}{2}. 
\end{align}
Let $\SC$ denote the class of functions satisfying the symmetric cone condition.
\end{assumption}

Our concern in this paper is to study the following open problems on \eqref{Seq:M}: 
\textit{solvability in the case of $\psi(\frac{1}{2})=c_*$, uniqueness, regularity}, 
under the assumption $\psi\in\SC$. 
Here 
\begin{align}\label{Seq:c_*}
c_* :=\frac{2}{c_0} = 0.8346262684\ldots
\end{align}
and $c_0$ is a constant given by 
\[c_0:=\int_{\R}\frac{1}{(1+t^2)^{\frac{5}{4}}}\dt =\mathcal{B}\Big(\frac{1}{2}, \frac{3}{4}\Big)=\sqrt{\pi}\frac{\Gamma(3/4)}{\Gamma(5/4)}= 2.396280469\ldots .
\]

With the argument in \cite{DD} we see that \eqref{Seq:M} has a solution if $\psi\in\SC$ satisfies $\psi(\frac{1}{2})<c_*$.
On the other hand, according to \cite{Muller01}, 
\eqref{Seq:M} has no solution if $\psi\in\SC$ satisfies $\psi(\frac{1}{2})>c_*$. 
Moreover, due to the lack of the convexity of $\W$, 
the uniqueness of solutions to \eqref{Seq:M} is an outstanding problem. 
In this paper, we are interested in the following: 
\begin{enumerate}
\item[(i)] \textit{Is problem \eqref{Seq:M} solvable under the assumption $\psi(\frac{1}{2})=c_*$?}
\item[(ii)] \textit{Does the uniqueness of solutions to problem \eqref{Seq:M} hold?}
\end{enumerate}
\noindent
If $u$ is a minimizer of $\W$ without obstacle, $u$ satisfies the equation on $(0,1)$
\begin{align}\label{Seq:1.1}
\frac{1}{\sqrt{1+(u'(x))^2}} \frac{d}{dx}\left( \frac{\vk_{u}'(x)}{\sqrt{1+(u'(x))^2}} \right) +\frac{1}{2}\vk_{u}(x)^3 =0
\end{align}
and the regularity of solutions of \eqref{Seq:1.1} is expected to be improved up to $C^{\infty}(0,1)$. 
However, the obstacle prevents us from applying arguments for \eqref{Seq:1.1} to problem~\eqref{Seq:M}.
We are also interested in the question
\begin{enumerate}
\item[(iii)]\textit{whether the regularity can be improved up to $C^{\infty}(0,1)$ or not.}
\end{enumerate} 
Although Dall'Acqua and Deckelnick have shown in \cite{DD} that third (weak) derivative of
solutions to \eqref{Seq:M} is of bounded variation in $(0,1)$,
it is not clear that solutions to \eqref{Seq:M} do not belong to $C^\infty(0,1)$ in general. 


We are ready to state our main result of this paper: 

\begin{theorem} \label{Sthm:1.1} 
Let $\psi\in\SC$ satisfy 
\begin{align}\label{Stop}
\psi\half<c_* .
\end{align}
Then problem \eqref{Seq:M} has a unique solution $u$.
In addition, $u\in C^2([0,1])$ holds and the third derivative of $u$ belongs to $BV(0,1)$, while
\begin{align}\label{Seq:1.2}
u\notin C^3([0,1]).
\end{align}
On the other hand, if $\psi\in\SC$ satisfies 
\begin{align} \label{Seq:1.7}
\psi\half\geq c_*, 
\end{align}
then \eqref{Seq:M} has no solution. 
\end{theorem}
Recently, Miura \cite{Miura2021} obtained the same uniqueness result in a different way; he focuses on the curvature and the proof is more geometric.
In Theorem~\ref{Sthm:1.1}, we restrict the class of obstacle to $\SC$ for a simplicity. 
For a more general assumption on $\psi$, see Remark~\ref{Srem:0726}. 
The reason why we employ Assumption~\ref{Sassumption:1.1} is to reduce problem \eqref{Seq:M} into the boundary value problem
\begin{align}\tag{BVP}\label{SBVP}
\begin{cases}
\frac{1}{\sqrt{1+(u'(x))^2}} \frac{d}{dx}\left( \frac{\vk_{u}'(x)}{\sqrt{1+(u'(x))^2}} \right) +\frac{1}{2}\vk_{u}(x)^3 =0 \quad \text{in}\quad 0<x<\frac{1}{2}, \\
u(0)=0,\quad u''(0)=0, \\
u(\frac{1}{2})=\psi(\tfrac{1}{2}), \quad u'(\tfrac{1}{2})=0,
\end{cases}
\end{align}
more precisely, see Section~\ref{Ssection:3}.
Hence we can obtain details of solutions to \eqref{Seq:M} via the shooting method, which is a method to know the properties of solutions of boundary value problems (see e.g.\  \cite{GG_06,NT_04,Pan_09,Schaaf}). 
As in \cite{DD, Muller01}, the study on problem \eqref{Seq:M} is done by variational approaches. 
One of novelties of this paper is to give another strategy, which makes use of the shooting method.
Furthermore, the shooting method enables us not only to give the complete classification of existence and non-existence of solutions to \eqref{Seq:M}, 
but also to make the graph of \eqref{Seq:M} via MAPLE since we can regard \eqref{Seq:M} as the Cauchy problem (see Figure~\ref{fig:1}).

\begin{figure}[htbp]
\centering
\includegraphics[width=4cm]{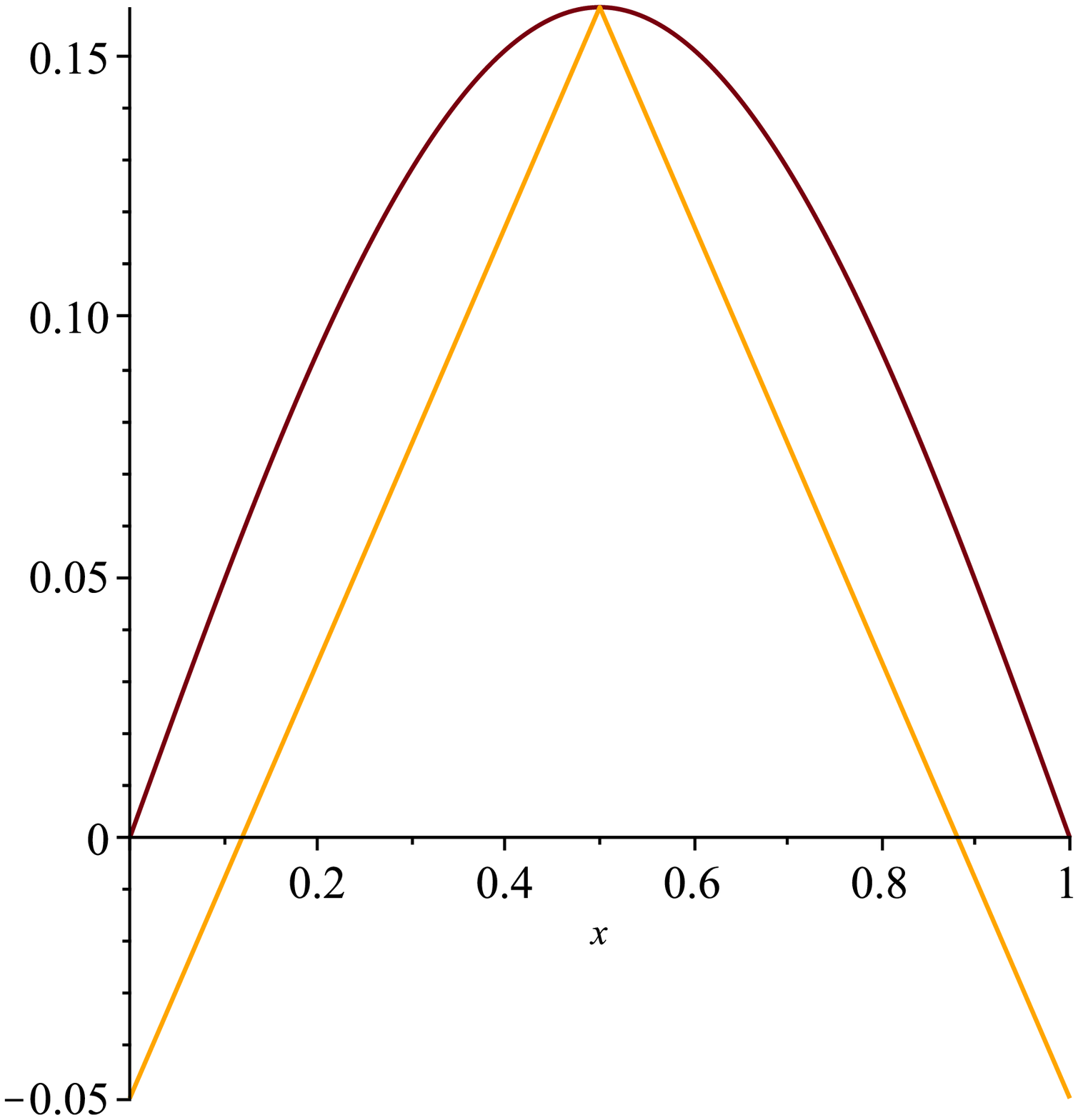} 
\hspace{2cm}
\includegraphics[width=4cm]{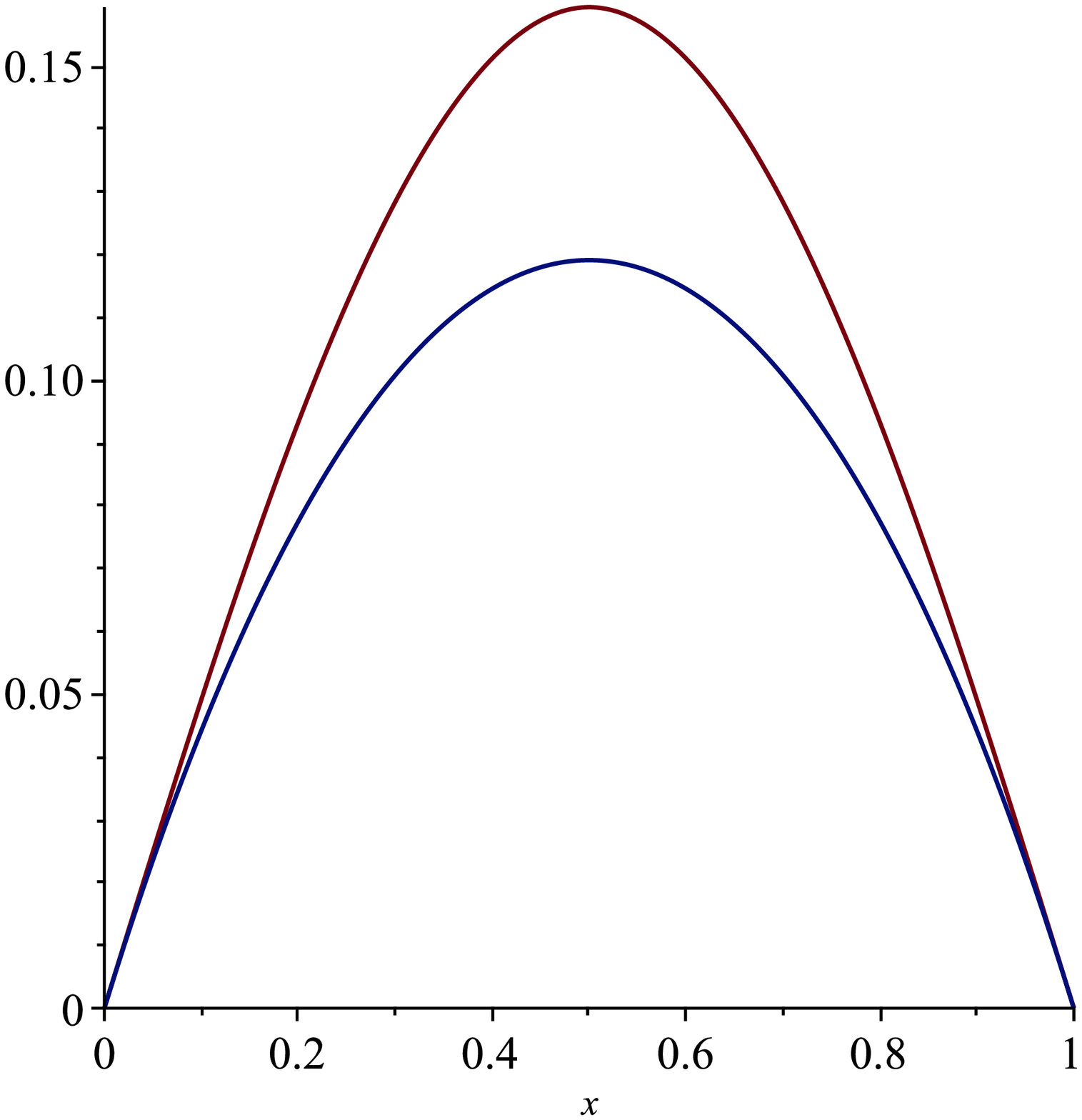}
\caption{The solution of \eqref{Seq:M} for the obstacle $\psi(\frac{1}{2})\approx0.16$ (left). 
The top line is the solution $u$ of \eqref{Seq:M} with $u'(0)=0.5$ and the other is the symmetric solution $v$ of \eqref{Seq:1.1} with $v(0)=0$ and $v'(0)=0.5$ (right).}
\label{fig:1}
\end{figure}

The shooting method is a useful tool to analyze the second order 
differential equations (see e.g.\  \cite{Azz, Bren, Brub, Pel}). 
On the other hand, it is not a standard matter to apply the shooting method to fourth order problems. 
Indeed, equation \eqref{Seq:1.1} is a quasilinear fourth order equation. 
However, using a geometric structure of \eqref{Seq:1.1}, we can reduce \eqref{Seq:1.1} into a second order 
semilinear equation. 
Then the standard shooting argument can work well for our problem. 
The reduction strategy is expected to be applicable to other fourth order geometric equations.

In this paper we are also interested in the ``variational inequality''. 
For a solution $u$ of \eqref{Seq:M}, $u+\ve(v-u)$ also belongs to $\Msym$ for $v\in \Msym$ and for $\ve\in [0,1]$ by the convexity of $\Msym$. 
Then using the minimality of $u$, we have
\begin{align*}
\frac{d}{d\ve} \W(u+\ve(v-u))\Big|_{\ve=0}\geq 0, 
\end{align*}
which leads the inequality 
\begin{align}\label{Seq:1.4}
\W'(u)(v-u) 
 \geq 0 \quad \text{for}\quad  v\in \Msym,
\end{align}
where $\W'(u)(\vp)$ denotes the first variation of $\W$ at $u$ in direction $\vp$ given by
\begin{align}\label{Seq:1.5}
\W'(u)(\vp) 
=\int_0^1\bigg[ 2\frac{u''\vp''}{(1+(u')^2)^{\frac{5}{2}}} -5\frac{(u'')^2u'\vp'}{(1+(u')^2)^{\frac{7}{2}}} \bigg]\dx.
\end{align}
Hence we see that a solution $u$ of \eqref{Seq:M} also solves the following problem:
\begin{align}\label{Seq:P}\tag{P}
\text{find}\quad u\in\Msym\quad \text{such that}\quad \W'(u)(v-u)  \geq 0 \quad \text{for all}\quad v\in \Msym.
\end{align}

We can obtain the same results on \eqref{Seq:P} as well as Theorem \ref{Sthm:1.1}: 

\begin{theorem} \label{Sthm:1.2} 
Assume that $\psi\in\SC$. 
Then each of the following holds$:$
\begin{itemize}
\item[(i)] If \eqref{Stop} holds, then \eqref{Seq:P} has a unique solution, which is obtained in Theorem~\ref{Sthm:1.1}$;$
\item[(ii)] If \eqref{Seq:1.7} holds, then \eqref{Seq:P} has no solution.
\end{itemize}
\end{theorem}
Since every solution of \eqref{Seq:M} also solves \eqref{Seq:P}, we shall prove Theorem~\ref{Sthm:1.2} and then Theorem~\ref{Sthm:1.1} can be deduced as a corollary of Theorem~\ref{Sthm:1.2}. 
The uniqueness of solutions of \eqref{Seq:P} is so important that the minimizer of $\W$ in $\Msym$ can be also characterized as the equilibrium of the corresponding parabolic problem (see Section~\ref{Ssection:5}).
Very recently M\"uller \cite{Mar20X} also considered \eqref{Seq:P} and the corresponding parabolic problem under a bit different assumption on $\psi$.
In \cite{Mar20X} he approached problem \eqref{Seq:P} by another way which is based on Talenti's symmetrization.

This paper is organized as follows: 
In Section~\ref{Ssection:2}, we collect notation and known results which are used in this paper. 
In Section~\ref{Ssection:3}, we identify coincidence sets of solutions to \eqref{Seq:P} and prove the concavity of solutions. 
In Section~\ref{Ssection:4}, we prove the uniqueness and regularity of \eqref{Seq:P}, using the shooting method. 
Finally, we apply these results to a parabolic problem in Section~\ref{Ssection:5}: we show that the solution of the corresponding dynamical problem converges to the solution of \eqref{Seq:M}.

\subsection*{Acknowledgments}
The author would like to thank Professor Shinya Okabe for fruitful discussions. 
The author would also like to thank referees for their careful reading and useful comments.
The author was supported by JSPS KAKENHI Grant Number 19J20749.


\section{Preliminaries} \label{Ssection:2}

In this section, we collect notation and some known properties on \eqref{Seq:M} and \eqref{Seq:P}. 

First, we see the relationship between \eqref{Seq:1.1} and the first variation of $u$. 
For $\vp\in C^{\infty}_{\rm c}(0,1)$ and a sufficiently smooth $u$, it follows from integration by parts that
\begin{align}\label{Seq:2.02}
\begin{split}
\W'(u)(\vp) 
&=\int_0^1\bigg( 2\frac{u''}{(1+(u')^2)^{\frac{5}{2}}}\vp'' -5\frac{(u'')^2u'}{(1+(u')^2)^{\frac{7}{2}}}\vp' \bigg)\dx \\
&=\int_0^1\bigg( -2\frac{u'''}{(1+(u')^2)^{\frac{5}{2}}} +5\frac{(u'')^2u'}{(1+(u')^2)^{\frac{7}{2}}} \bigg)\vp'\dx \\
&= \int_0^1 2\left(\frac{1}{\sqrt{1+(u')^2}} \frac{d}{dx}\bigg( \frac{\vk_{u}'}{\sqrt{1+(u')^2}} \bigg) +\frac{1}{2}\vk_{u}^3 \right)\vp \dx. 
\end{split}
\end{align}

\begin{lemma}
$u$ is a solution of \eqref{Seq:P} if and only if $u$ solves
\begin{align}\label{Seq:P'}
\text{find}\quad u\in\Msym\quad \text{such that}\quad \W'(u)(v-u)  \geq 0 \quad \text{for all}\quad v\in M,
\end{align}
where  $M$ is the convex set of $H(0,1)$ defined by 
\[
M:=\Set{v\in H(0,1) | v\geq \psi \ \ \text{in} \ \ [0,1] }.
\]
\end{lemma}
\begin{proof}
Since we see at once that the sufficiency of \eqref{Seq:P'} is clear, we show the necessity of \eqref{Seq:P'}. 
Let $u$ be a solution of \eqref{Seq:P} and fix $v\in M$ arbitrarily. 
Set
\[ v_1(x):= \frac{1}{2}\Big( v(x)+v(1-x)\Big), \quad v_2(x):=v(x)-v_1(x). \]
Then we find $v_1\in\Msym$. 
Taking $v_1$ as the test function in \eqref{Seq:1.4}, we have
\[ 0\leq \W'(u)(v_1-u)=\W'(u)(v-u)-\W'(u)(v_2)\]
and hence it suffices to show $\W'(u)(v_2)=0$. 
Since $u\in\Msym$ and $v_2$ satisfies $v_2(1-x)=-v_2(x)$, in view of \eqref{Seq:1.5} we see that
\begin{align*}
\int_{\frac{1}{2}}^1\bigg( 2\frac{u''v_2''}{(1+(u')^2)^{\frac{5}{2}}} &-5\frac{(u'')^2u'v_2'}{(1+(u')^2)^{\frac{7}{2}}} \bigg)\dx\\
&=-\int_0^{\frac{1}{2}}\bigg( 2\frac{u''v_2''}{(1+(u')^2)^{\frac{5}{2}}} -5\frac{(u'')^2u'v_2'}{(1+(u')^2)^{\frac{7}{2}}} \bigg)\dx,
\end{align*}
which clearly yields $\W'(u)(v_2)=0$.
Therefore we obtain 
\begin{align}\label{Seq:2.2}
\W'(u)(v-u) 
 \geq 0 \quad \text{for}\quad  v\in M.
\end{align}
The proof is complete.
\end{proof}

Next we define a useful function $G$ introduced in \cite{DG_07}. 
Let $G:\R\to(-\tfrac{c_0}{2}, \tfrac{c_0}{2})$ be 
\begin{align}\label{Sdef-G}
 G(x):=\int_0^x \frac{1}{(1+t^2)^{\frac{5}{4}}} \dt,    
\end{align}
where $c_0:=\int_{\R}(1+t^2)^{-\frac{5}{4}}\dt$. 
The function $G$ is bijective and strictly increasing
since $G'(x)>0$.
Therefore $G^{-1}$ exists, is smooth, and satisfies
\begin{align}\label{Seq:1.05}
\dfrac{d}{dx}G^{-1}(x) = \Big( 1+G^{-1}(x)^2 \Big)^{\frac{5}{4}}.
\end{align}

\begin{proposition}[Known results on \eqref{Seq:M}, \cite{DD}] \label{Sprop:2.1} 
Assume that $\psi\in\SC$ and let $c_*$ be the constant given by \eqref{Stop}. 
Then the following hold$:$
\begin{itemize}
\item[(i)] If $\psi$ satisfies $\psi(\frac{1}{2})<c_*$, then \eqref{Seq:M} has a solution.
\item[(ii)] Solutions of \eqref{Seq:M} are concave.
\end{itemize}
\end{proposition}
These are shown in \cite[Section 4.1]{DD}. 

\begin{proposition}[Known results on \eqref{Seq:P}, \cite{DD}] \label{Sprop:3.1} 
Assume that $\psi$ satisfies 
\begin{align}\label{Seq:psi}
\psi\in C([0,1]), \quad \psi(0)<0, \ \  \psi(1)<0 \quad \text{and} \quad \max_{0\leq x \leq 1}\psi(x)>0
\end{align} 
and let $u$ be a solution of \eqref{Seq:P}. 
Then the following hold$:$

\smallskip

{\rm (i)} Suppose that $u(x)>\psi(x)$ for all $x\in E:=(x_1,x_2)\subset(0,1)$. 
\begin{itemize}
\item[(a)] $u\in C^{\infty}(\bar{E})$ and $u$ satisfies \eqref{Seq:1.1} on $E$. 
\item[(b)] $v(x):=\vk_{u}(x)(1+u'(x)^2)^{\frac{1}{4}}=\frac{u''(x)}{(1+(u'(x))^2)^{\frac{5}{4}}}$ satisfies on $E$ 
\begin{align}\label{Seq:3.04}
-\frac{d}{dx}\left( \frac{v'(x)}{(1+(u'(x))^2)^{\frac{3}{4}}} \right) +\frac{\vk_{u}(x)u'(x)}{(1+(u'(x))^2)^{\frac{1}{4}}}v'(x) =0. 
\end{align}
\end{itemize}

\smallskip

{\rm (ii)} $\vk_{u}(0)=\vk_{u}(1)=0$, i.e., $u''(0)=u''(1)=0$. 

\smallskip

{\rm (iii)} Every solution $u$ of \eqref{Seq:P} satisfies
\begin{align}\label{SBV}
u\in C^2([0,1]) \quad \text{and} \quad u'''\in BV(0,1).
\end{align}
\end{proposition}
The proofs of (i), (ii) and (iii) are given in \cite[Proposition 3.2]{DD}, \cite[Corollary 3.3]{DD} and \cite[Theorem 5.1]{DD}, respectively. 
If $\psi$ belongs to $\SC$, then $\psi$ clearly satisfies \eqref{Seq:psi}.

The representation of \eqref{Seq:3.04} is so important that the following comparison principle holds.
\begin{proposition} \label{Sprop:2.3} 
If $u$ satisfies \eqref{Seq:1.1} on $(x_1, x_2)$, then $v(x)=\frac{u''(x)}{(1+(u'(x))^2)^{\frac{5}{4}}}$ satisfies 
\begin{align} \label{Scom}
\min\{ v(x_1), v(x_2) \} \leq v(x) \leq \max\{v(x_1), v(x_2)\} \quad \text{for} \quad x_1<x< x_2.
\end{align}
\end{proposition}
For the proof we refer the reader to \cite{DG_07}.


\section{Concavity and coincidence set} \label{Ssection:3}

According to \cite[Proposition~3.2]{Muller01}, solutions of \eqref{Seq:M} touch $\psi$ only at $x=1/2$ under the assumption $\psi\in\SC$.
In this section  we deduce that solutions of \eqref{Seq:P} also touch $\psi$ only at $x=1/2$.
From this fact, we also show that solutions of \eqref{Seq:P} are concave. 
We remark here that the results in Section~\ref{Ssection:3} do not depend on the height of $\psi$.

Thanks to \eqref{SBV}, 
every solution $u$ of \eqref{Seq:P} satisfies $u'''\in BV(0,1)$. 
For $\vp\in C^{\infty}_{\rm c}(0,\frac{1}{2})$ with $\vp\geq0$, by using $v=u+\vp$ as a test function in \eqref{Seq:2.2}, we find that $u$  satisfies 
\begin{align}\label{Seq:3.01}
\int_0^{\frac{1}{2}} \bigg(-2\frac{u'''(x)}{(1+u'(x)^2)^{\frac{5}{2}}}+5 \frac{u''(x)^2u'(x)}{(1+u'(x)^2)^{\frac{7}{2}}}\bigg) \vp'(x)\dx \geq 0 
\end{align}
for any $0\leq \vp\in C^{\infty}_{\rm c}(0,\frac{1}{2})$.

\begin{lemma} \label{Sprop:3.2} 
Let $\psi\in \SC$ and $u$ be a solution of \eqref{Seq:P}.
Then $u$ satisfies $u(\tfrac{1}{2})=\psi(\tfrac{1}{2})$. 
Moreover, $u(x)\ne\psi(x)$ for $x\ne\tfrac{1}{2}$.
\end{lemma}
\begin{proof}
Fix $u$ as a solution of \eqref{Seq:P}. 
We divide the proof  into three steps.

\smallskip

\noindent
\textbf{Step 1}. {\sl $u$ touches $\psi$ at $x=\tfrac{1}{2}$.\ } 
Assume that $u(\tfrac{1}{2})>\psi(\tfrac{1}{2})$. 
First let us suppose that 
$$I:=\Set{x\in(0,\tfrac{1}{2}) | u(x)=\psi(x)} \ne \emptyset.$$ 
Then we can define $x_{1}:=\sup I \in(0,\tfrac{1}{2})$ such that $u(x_1)=\psi(x_1)$ and $u'(x_1)=\psi'(x_1)$. 
By symmetry we obtain 
\[u(x)>\psi(x) \quad \text{in} \quad x\in(x_{1},1-x_{1}). \]
Therefore by Propositions~\ref{Sprop:3.1} and \ref{Sprop:2.3}, $v(x)=u''(x)(1+u'(x)^2)^{-\frac{5}{4}}$ satisfies 
\[
\min\Set{v(x_1), v(1-x_1)} \leq v(x) \leq \max\Set{v(x_1), v(1-x_1)} \quad \text{for}\quad x\in(x_1, 1-x_1), 
\]
which implies that $v\equiv$ const.\ in $[x_1, 1-x_1]$ since $v(x_1)= v(1-x_1)$ holds. 
Then by the same argument as in \cite[Lemma 4]{DG_07}, there exists $c\in (-\tfrac{c_0}{2}, \tfrac{c_0}{2})$ such that $u$ satisfies
\begin{align}\label{Seq:3.2}
u'(x)=G^{-1}\Big(\frac{c}{2}- cx\Big) \quad \text{in} \quad[x_1, 1-x_1]. 
\end{align}
Here we note that the curve $u$ given by \eqref{Seq:3.2} is concave for all $c$. 
However, taking account of the shape of $\psi\in\SC$, $u'(x_1)=\psi'(x_1)$, and the concavity of \eqref{Seq:3.2}, we find the contradiction to $u>\psi$ in $(x_{1},1-x_{1})$ 
(see e.g.\ Figure~\ref{Sfig:12}).

Next, if $I=\emptyset$, Proposition~\ref{Sprop:3.1}(i) implies that $u$ satisfies $u\in C^{\infty}([0,1])$ and that $u$ satisfies \eqref{Seq:1.1} on $(0,1)$.
Moreover, it follows from Proposition~\ref{Sprop:3.1}(ii) that $u$ also satisfies $u''(0)=0$.
However, according to  \cite[Theorem 1]{DG_07}, such $u$ is limited to $u\equiv0$, 
which contradicts to $u\geq\psi$. 
Hence no matter whether $I$ is empty or not, $u(\tfrac{1}{2})=\psi(\tfrac{1}{2})$ holds.

\begin{figure}[htbp]
\begin{center}
\includegraphics[width=5.5cm]{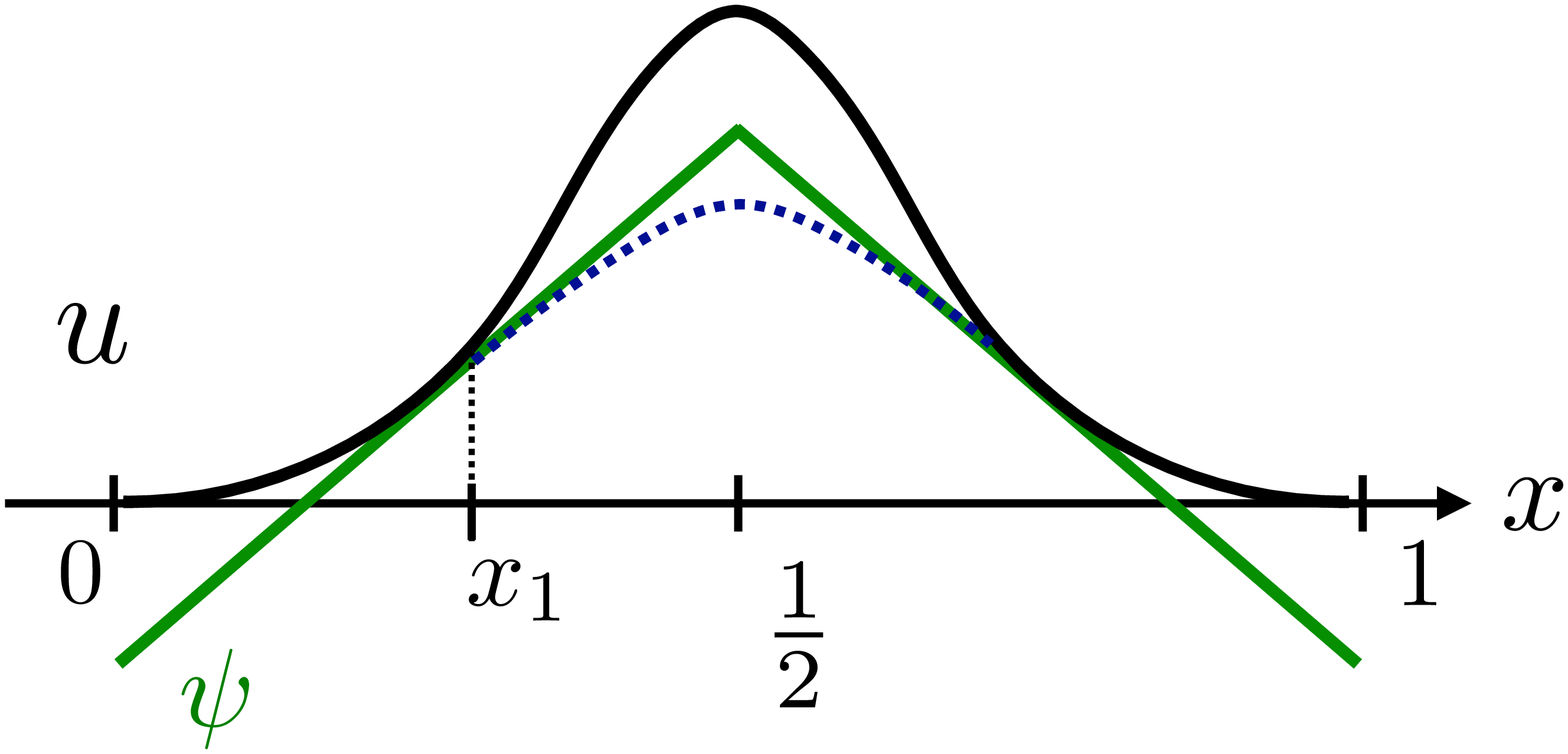}
\quad
\includegraphics[width=5.5cm]{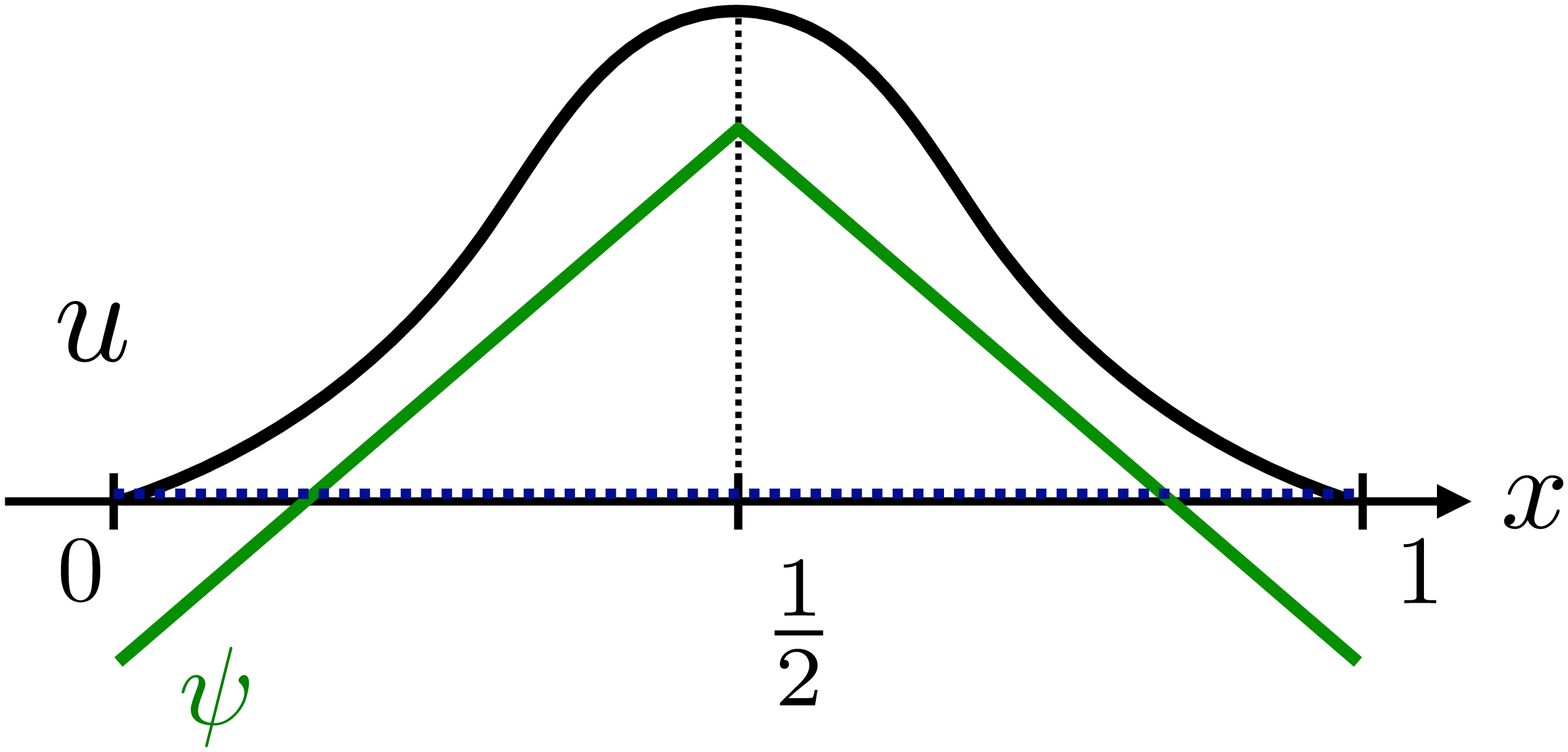}
\caption{When $I$ is non-empty, $u|_{[x_1,1-x_1]}$ must be the dotted curve (left), while $u$ must be trivial when $I$ is empty (right).
\label{Sfig:12}}
\end{center}
\end{figure}


\noindent
\textbf{Step 2}. {\sl The coincidence set has zero Lebesgue measure.\ } 
Let $$N:=\{x\in(0,\tfrac{1}{2})\, |\, u(x)>\psi(x)\}.$$
If  there exist $x_2, x_3 \in (0,\tfrac{1}{2})$ such that $(x_2, x_3)$ is a connected component of $N$, then it follows from Proposition~\ref{Sprop:3.1}(i) that $u$ satisfies \eqref{Seq:1.1} on $(x_2, x_3)$. 
By Proposition~\ref{Sprop:2.3}, $v=u''(1+(u')^2)^{-\frac{5}{4}}$ satisfies 
\begin{align}\label{Seq:3.05}
v(x)\geq \min\{ v(x_2), v(x_3)\} \quad \text{for}\quad x\in (x_2, x_3).
\end{align}
Moreover, since $u-\psi$ attains minimum at $x=x_2, x_3$, it holds that $u''(x_2)$, $u''(x_3)\geq 0$, which in combination with \eqref{Seq:3.05} gives $u''\geq0$ in $[x_2, x_3]$. 
However, $u''\geq0$ in $[x_2, x_3]$ and $u'(x_2)=u'(x_3)$ imply that $u\equiv\psi$ in $[x_2, x_3]$, which contradicts the fact that $(x_2, x_3)$ is a connected component of $N$. 

Next let us suppose that there are $0<x_4 < x_5 <{1}/{2}$ such that $(0,x_4)$ and $(x_5,\tfrac{1}{2})$ are connected components of $N$, respectively.
The previous argument implies that $u\equiv\psi$ in $[x_4, x_5]$ and hence that $u''=0$ in $[x_4, x_5]$ 
since $\psi''\equiv0$. 
Moreover, recalling  that $v=u''(1+(u')^2)^{-\frac{5}{4}}$ satisfies the comparison principle \eqref{Scom} on $(x_5, \tfrac{1}{2})$, we find that one of the following holds: 
\begin{align}\label{Seq:3.4}
\begin{cases}
\text{(i)}\quad  &u''(\tfrac{1}{2})\geq 0 \quad \text{and} \quad u''(x)\geq 0 \quad\text{in}\quad [x_5, \tfrac{1}{2}];\\
\text{(ii)} &u''(\tfrac{1}{2})< 0 \quad \text{and} \quad u''(x)\leq 0 \quad\text{in}\quad [x_5, \tfrac{1}{2}].
\end{cases}
\end{align}
In the case of (i), it holds that $u''\geq 0$ in $[x_5, \frac{1}{2}]$ and $u'(x_5)>0$, which contradicts $u'(\frac{1}{2})=0$. 
Supposing (ii), we infer from $u'(x_5)=\psi'(x_5)$ that $u(\frac{1}{2})<\psi(\frac{1}{2})$.
Therefore both (i) and (ii) do not occur and 
we conclude that either 
$ N=(0,\frac{1}{2}) $
or
\begin{align}\label{Seq:3.5}
 N=(0,a)\cup\Big(a,\frac{1}{2}\Big) \quad \text{for some} \quad 0<a<\frac{1}{2}
\end{align}
holds.



\noindent
\textbf{Step 3}. {\sl We show that $u(x)>\psi(x)$ if $x\ne\tfrac{1}{2}$.\ } 
It is sufficient to show  that \eqref{Seq:3.5} does not occur.
Suppose, to the contrary, that there exists $a\in (0,\tfrac{1}{2})$ satisfying $u(a)=\psi(a)$. 
Then $u''(a)\geq 0$ since $u-\psi$ attains the minimum at $x=a$. 
Moreover, we have
\begin{align}\label{Seq:3.06}
u''(a)> 0. 
\end{align}
In fact, if $u''(a)=0$, then we obtain the same contradiction as in \eqref{Seq:3.4}. 
Let us recall that
\begin{align*}
u(0)=0, \quad u'(\tfrac{1}{2})=0 \quad \text{and}\quad u''(0)=0.
\end{align*}
Set $A_1:=(0,a)$ and $A_2:=(a,\tfrac{1}{2})$.
Then $u>\psi$ in $A_i$, $i=1,2$. 
Therefore Proposition~\ref{Sprop:3.1}(a) implies that $u\in C^{\infty}(\bar{A_i})$ and we have
\begin{align}\label{Seq:3.08}
 \bigg(-2\frac{u'''(x)}{(1+u'(x)^2)^{\frac{5}{2}}}+5 \frac{u''(x)^2u'(x)}{(1+u'(x)^2)^{\frac{7}{2}}}\bigg)'= 0 \quad \text{in}\quad A_{i}
\end{align}
for $i=1,2$, respectively.

First we focus on $A_1$. 
Since $u''(0)=0$ and $u''(a)>0$ hold, combining these with \eqref{Scom} we have $u''(x)\geq 0$ for $x\in A_1$. 
Hence $u'''(0)\geq0$. 
If $u'''(0)=0$, 
then $u$ satisfies \eqref{Seq:1.1} with $u(0)=u''(0)=u'''(0)=0$ and such $u$ is limited to a line segment.
This contradicts $u'(\tfrac{1}{2})=0$. 
Therefore $u'''(0)>0$, which in combination with \eqref{Seq:3.08} and $u''(0)=0$ gives
\begin{align}\label{Seq:3.09}
2\frac{u'''(x)}{(1+u'(x)^2)^{\frac{5}{2}}}-5 \frac{u''(x)^2u'(x)}{(1+u'(x)^2)^{\frac{7}{2}}}
=2\frac{u'''(0)}{(1+u'(0)^2)^{\frac{5}{2}}}=:\eta_1>0
\end{align}
for all $x\in (0,a)=A_1$.

Next we focus on $A_2$.
If $u''(\tfrac{1}{2})> 0$, then \eqref{Scom} and \eqref{Seq:3.06} imply that $u''(x)\geq 0$ for $x\in\bar{A_2}$. 
However this leads to a contradiction by the same method as in \eqref{Seq:3.4} and hence we have 
$u''(\frac{1}{2})\leq 0. $ 
Combining the comparison principle for $v(x)=u''(x)(1+u'(x)^2)^{-\frac{5}{4}}$ with $v(a)>0$ and $v(\tfrac{1}{2})\leq0$, we find that $v$ attains the minimum at $x=\tfrac{1}{2}$ in $a\leq x \leq \tfrac{1}{2}$. 
Then it holds that 
\[ 
v'(x)\Big|_{x=\frac{1}{2}} = \frac{u'''(x)}{(1+u'(x)^2)^{\frac{5}{4}}}-\frac{5}{2} \frac{u''(x)^2u'(x)}{(1+u'(x)^2)^{\frac{9}{4}}}\Bigg|_{x=\frac{1}{2}} \leq0.
\]
This together with $u'(\frac{1}{2})=0$ implies that $u'''(\tfrac{1}{2})\leq0$. 
Following the same way as in \eqref{Seq:3.09}, we infer from \eqref{Seq:3.08} that
\begin{align}\label{Seq:3.10}
2\frac{u'''(x)}{(1+u'(x)^2)^{\frac{5}{2}}}-5 \frac{u''(x)^2u'(x)}{(1+u'(x)^2)^{\frac{7}{2}}}
=2\frac{u'''(\tfrac{1}{2})}{(1+u'(\tfrac{1}{2})^2)^{\frac{5}{2}}}=:\eta_2 \leq 0 
\end{align}
for all $x\in (a,\frac{1}{2})=A_2$.

Finally, fix $0\leq \vp\in C^{\infty}_{\rm c}(0,\frac{1}{2})$ such that $\vp(a)>0$. 
Then the left-hand side of \eqref{Seq:3.01} is reduced into
\begin{align*}
\int_0^{\frac{1}{2}} \bigg(-2\frac{u'''(x)}{(1+u'(x)^2)^{\frac{5}{2}}}&+5 \frac{u''(x)^2u'(x)}{(1+u'(x)^2)^{\frac{7}{2}}}\bigg) \vp'(x)\dx \\
&=\int_0^a -\eta_1\vp'(x) \dx + \int_a^{\frac{1}{2}}-\eta_2\vp'(x) \dx \\
&=(-\eta_1+\eta_2)\vp(a) <0, 
\end{align*}
where we used \eqref{Seq:3.09} and \eqref{Seq:3.10}. 
However, this clearly contradicts \eqref{Seq:3.01}. 
The proof is complete.
\end{proof}

By Proposition~\ref{Sprop:3.1}, Lemma~\ref{Sprop:3.2}, and the symmetry property of $M_{\rm sym}$, solutions of \eqref{Seq:P} satisfy \eqref{SBVP}.
From this fact, we deduce the concavity of solutions of \eqref{Seq:P}. 

\begin{proposition} \label{Sprop:3.3} 
Let $\psi\in \SC$.
Then every solution $u$ of \eqref{Seq:P} is concave if it exists. 
Moreover, it holds that
$$u''\Big(\frac{1}{2}\Big)<0.$$
\end{proposition}
\begin{proof}
Recall that all solutions of \eqref{Seq:P} satisfy \eqref{SBVP}.
Let $u$ be a solution of \eqref{SBVP}. 
It suffices to show that $u''(\tfrac{1}{2})<0$ since we infer from Proposition~\ref{Sprop:2.3}(i) that if $u''(\tfrac{1}{2})<0$, then
\begin{align*}
\frac{u''(x)}{(1+(u'(x))^2)^{\frac{5}{4}}} =v(x) \leq \max\Set{v(0), v(\tfrac{1}{2})}=0 
\quad \text{for}\quad x\in\Big(0,\frac{1}{2}\Big).
\end{align*} 
This clearly implies that $u''(x)\leq0$ in $[0,\tfrac{1}{2}]$, and $u''(x)\leq0$ in $[\tfrac{1}{2},1]$ also holds by symmetry.

Suppose, to the contrary, that $u''(\tfrac{1}{2})\geq0$. 
Then comparison principle \eqref{Scom} implies that
\begin{align*}
\frac{u''(x)}{(1+(u'(x))^2)^{\frac{5}{4}}} =v(x) \geq \min\Set{v(0), v(\tfrac{1}{2})}=0 
\quad \text{for} \quad x\in\Big(0,\frac{1}{2}\Big).
\end{align*}
Hence $u''(x)\geq0$ in $[0,\tfrac{1}{2}]$. 
However, such $u$ does not satisfy $u'(\tfrac{1}{2})=0$ since $u(\frac{1}{2})=\psi(\frac{1}{2})$.
This contradicts our assumption and hence we obtain $u''(\frac{1}{2})<0$. 
The proof is complete.
\end{proof}

\begin{remark} \label{Srem:0726}
Thanks to Proposition~\ref{Sprop:2.1}(ii), solutions of \eqref{Seq:M} are concave and hence \eqref{Seq:M} 
can be reduced to \eqref{SBVP} under the assumption that
\begin{align}\label{S0726-2}
\psi\in C^2([0,\tfrac{1}{2}]) \quad \text{satisfies} \quad \psi'(0)\geq0 
\quad \text{and} \quad \psi''\geq0 \ \ \text{in}\ \ [0,\tfrac{1}{2}],
\end{align}
which includes \eqref{S0726-1}.
Then, from the argument in Section~\ref{Ssection:4}, we can prove Theorem~\ref{Sthm:1.1} under assumption \eqref{S0726-2} instead of \eqref{S0726-1}. 
On the other hand, in general, problem \eqref{Seq:P} cannot be reduced to \eqref{SBVP} under assumption \eqref{S0726-2}. 
Indeed, it is not so clear that Lemma~\ref{Sprop:3.2} holds under assumption \eqref{S0726-2}.
\end{remark}


\section{Shooting method}\label{Ssection:4}

Due to the arguments in Section~\ref{Ssection:3}, solutions of \eqref{Seq:P} satisfy \eqref{SBVP}. 
In Subsection~\ref{Ssubsec:4.1} we show some properties of solutions of \eqref{SBVP}. 
Applying them, we show the uniqueness and the regularity of the solution of \eqref{Seq:P} in Subsection~\ref{Ssubsec:4.2}.

\subsection{Two-point boundary value problem}\label{Ssubsec:4.1}

In this subsection, we consider the multiplicity of solutions to \eqref{SBVP}, that is, 
\begin{align*}
\frac{1}{\sqrt{1+(u'(x))^2}} \frac{d}{dx}\left( \frac{\vk_{u}'(x)}{\sqrt{1+(u'(x))^2}} \right) +\frac{1}{2}\vk_{u}(x)^3 =0, \quad  0< x <\frac{1}{2}
\end{align*}
with the boundary conditions
\begin{align}\label{Seq:BC}
u(0)=0, \quad u''(0)=0, \quad u\Big(\frac{1}{2}\Big)=\psi\Big(\frac{1}{2}\Big), \quad 
u'\Big(\frac{1}{2}\Big)=0. 
\end{align}

We shall show that \eqref{SBVP} has at most one solution, using the shooting method. 
In the following we consider the initial condition 
\begin{align}\label{Seq:IC}
u(0)=0, \quad u'(0)=\va, \quad u''(0)=0, \quad u'''(0)=\vb
\end{align}
instead of the boundary condition \eqref{Seq:BC} and find the condition of $(\va,\vb)$ to satisfy $u(\tfrac{1}{2})=\psi(\tfrac{1}{2})$ and $u'(\tfrac{1}{2})=0$.
Since it follows from Proposition~\ref{Sprop:3.3} that solutions of \eqref{Seq:P} are concave, we only focus on $u'(0)=\va>0$.

To begin with, as discussed in \eqref{Seq:2.02}, we can reduce \eqref{Seq:1.1} into
\begin{align}\label{Seq:4.3}
\left(2\frac{u'''(x)}{(1+u'(x)^{2})^{\frac{5}{2}}}-5\frac{u''(x)^2u'(x)}{(1+u'(x)^{2})^{\frac{7}{2}}} \right)'=0.
\end{align}
Let us set
\[ u'(x)=:w(x). \]
Then by \eqref{Seq:4.3} $w$ satisfies 
\begin{align}\label{Seq:4.04}
\begin{split}
\frac{w''(x)}{(1+w(x)^2)^{\frac{5}{2}}}-\frac{5}{2}\frac{w'(x)^2w(x)}{(1+w(x)^{2})^{\frac{7}{2}}} 
&=\frac{w''(0)}{(1+w(0)^2)^{\frac{5}{2}}}-\frac{5}{2}\frac{w'(0)^2w(0)}{(1+w(0)^{2})^{\frac{7}{2}}}\\
&= \dfrac{\vb}{(1+\va^2)^{\frac{5}{2}}},
\end{split}
\end{align}
where we used the initial data \eqref{Seq:IC}. 
Using $G$, which is defined in \eqref{Sdef-G}, we set
\begin{align}\label{Seq:set-y}
y(x): =G(w(x)) = G(u'(x)). 
\end{align}
Then combining \eqref{Seq:4.04} and \eqref{Seq:set-y} with
\begin{align*}
y'(x) = \frac{u''(x)}{(1+u'(x)^2)^{\frac{5}{4}}}, \quad
y''(x)=\frac{u'''(x)}{(1+u'(x)^2)^{\frac{5}{4}}}-\frac{5}{2}\frac{u''(x)^2u'(x)}{(1+u'(x)^2)^{\frac{9}{4}}}, 
\end{align*}
we obtain 
\begin{align}\label{Seq:4.05}
\dfrac{y''(x)}{\big(1+G^{-1}(y(x))^2\big)^{\frac{5}{4}}} = \dfrac{\vb}{(1+\va^2)^{\frac{5}{2}}}. 
\end{align}
By $w(0)=\va$ and $w'(0)=0$, we consider 
\begin{align}\label{Seq:4.06}
 y(0)=G(\va)>0 \quad\text{and}\quad  y'(0) =\dfrac{w'(0)}{(1+w(0)^2)^{\frac{5}{4}}}=0
\end{align}
as the initial data for \eqref{Seq:4.05}.
If $u$ satisfies $u'(\tfrac{1}{2})=0$, then $y$ must attain zero at $x=\tfrac{1}{2}$.  
Therefore at first we seek the condition $(\va,\vb)$ to satisfy $y(\tfrac{1}{2})=0$.
By the representation of \eqref{Seq:4.05} and the initial condition \eqref{Seq:4.06}, 
we notice that 
\begin{align*}
\vb<0
\end{align*}
if and only if the solution $y$ has zero. 
Furthermore, if $\vb<0$ then zero of $y$ is unique since $\vb<0$ implies that $y''(x)<0$. 
The point where $y$ achieves zero, which is often called \textit{time map formula}, is given as follows:

\begin{lemma} \label{Sprop:4.1} 
Let $(\va,\vb)\in\R_{>0}\times\R_{<0}$ be arbitrary and $y$ be the solution of \eqref{Seq:4.05} with \eqref{Seq:4.06}. 
The point where $y$ achieves zero is given by
\begin{align}\label{Stime-map}
Z_{\va,\vb}
=\dfrac{(1+\va^2)^{\frac{5}{4}}}{\sqrt{2}\sqrt{|\vb|}}\int_0^{\va} \dfrac{1}{\sqrt{\va-t}} \dfrac{\dt}{(1+t^2)^{\frac{5}{4}}}.
\end{align}
\end{lemma}
\begin{proof}
Fix $(\va,\vb)\in\R_{>0}\times\R_{<0}$ arbitrarily and let $y$ be the solution of \eqref{Seq:4.05} with \eqref{Seq:4.06}.
Let us define
\[ f(t):= (1+G^{-1}(t)^2)^{\frac{5}{4}}, \quad F(X):=2\int_0^X f(t)\dt. \]
Set $Z\in(0,\infty)$ as the point where $y$ achieves zero.
Then we deduce from \eqref{Seq:4.05} that
\begin{align*}
\dfrac{d}{dx} \bigg( y'(x)^2 - \dfrac{\vb}{(1+\va^2)^{\frac{5}{2}}}F(y(x)) \bigg) 
&= 2y' \bigg( y'' -\dfrac{\vb}{(1+\va^2)^{\frac{5}{2}}}\big(1+G^{-1}(y)^2\big)^{\frac{5}{4}} \bigg)
=0, 
\end{align*}
which in combination with \eqref{Seq:4.06} gives  
\begin{align}\label{Seq:4.09}
y'(x)^2 - \dfrac{\vb}{(1+\va^2)^{\frac{5}{2}}}F(y(x)) = -\dfrac{\vb}{(1+\va^2)^{\frac{5}{2}}}F(y(0))
\quad \text{for}\quad x\in (0,Z).
\end{align}
Moreover, since $y'(0)=0$ and $y''<0$ in $(0,Z)$, we find $y'(x)<0$ in $(0,Z)$.
Combining this with \eqref{Seq:4.09},  we obtain 
\begin{align}\label{Seq:1103-3}
\begin{split}
y'(x) &=- \dfrac{\sqrt{|\vb|}}{(1+\va^2)^{\frac{5}{4}}}\sqrt{F(G(\va))-F(y(x))}\\
&=- \dfrac{\sqrt{|\vb|}}{(1+\va^2)^{\frac{5}{4}}}\sqrt{2\va-2G^{-1}(y(x))}.
\end{split}
\end{align}
Here we used   
\begin{align*}
F(s) = 2\int_0^s \big(1+G^{-1}(t)^2\big)^{\frac{5}{4}} \dt = 2\int_0^s \big( G^{-1}(t)\big)' \dt = 2G^{-1}(s),
\end{align*}
which follows from \eqref{Seq:1.05} and $G^{-1}(0)=0$. 
Integrating \eqref{Seq:1103-3} on $(0,Z)$, we obtain 
\begin{align*}
\dfrac{\sqrt{|\vb|}}{(1+\va^2)^{\frac{5}{4}}} Z
&=\int_0^{Z} -\dfrac{y'(t)}{\sqrt{2\va-2G^{-1}(y(t))}}\dt
=\int_0^{G(\va)} \dfrac{\ds}{\sqrt{2\va-2G^{-1}(s)}}, 
\end{align*}
where we used the change of the variables $s=y(t)$ in the last equality. 
Therefore we have
\begin{align*}
Z
=\dfrac{(1+\va^2)^{\frac{5}{4}}}{\sqrt{|\vb|}}\int_0^{G(\va)} \dfrac{\ds}{\sqrt{2\va-2G^{-1}(s)}}.
\end{align*}
By the change of variables $t=G^{-1}(s)$, we obtain \eqref{Stime-map}.
\end{proof}

By Lemma~\ref{Sprop:4.1}, for each $\va>0$ the map on $\R_{<0}$
\[ \vb \mapsto Z_{\va,\vb}\]
is strictly increasing and satisfies
\[ \lim_{\vb\uparrow0}Z_{\va,\vb}=\infty \quad \text{and}\quad \lim_{\vb\to-\infty}Z_{\va,\vb}=0.\]
Therefore \eqref{Stime-map} implies that for each $\va>0$ there exists a unique $\vb<0$ such that $Z_{\va,\vb}=1/2$. 
Replacing $Z_{\va,\vb}$ with $1/2$ in \eqref{Stime-map}, we obtain the following.

\begin{proposition} \label{Sprop:4.2} 
For each $\va>0$, 
$Z_{\va,\vb}=1/2$ holds if and only if 
\begin{align}\notag
\vb=\vb_*(\va):= -2(1+\va^2)^{\frac{5}{2}}\left(\int_0^{\va}\dfrac{1}{\sqrt{\va-t}}\dfrac{\dt}{(1+t^2)^{\frac{5}{4}}} \right)^2.
\end{align}
\end{proposition}

Thus for $\va=u'(0)$, 
\[ u'''(0)=\vb_{*}(\va)\]
is needed in \eqref{Seq:IC} so that the solution $u$ of \eqref{Seq:1.1} with \eqref{Seq:IC} satisfies $u'(\tfrac{1}{2})=0$. 
Hence we should consider 
\begin{align}\label{Su_*}
u(0)=0, \quad u'(0)=\va, \quad u''(0)=0, \quad u'''(0)=\vb_*(\va). 
\end{align}

Next we investigate the relationship between $u'(0)=\va$ and $u(\tfrac{1}{2})=\psi(\tfrac{1}{2})$.
To this end, hereafter we consider only the case $\vb=\vb_{*}(\va)$ in \eqref{Seq:IC}. 
Let $u(x;\va)$ denote the solution of \eqref{Seq:1.1} with \eqref{Su_*}. 
Then $y(x;\va):=G(u'(x;\va))$ is the solution of
\begin{align}\label{Sy_a}
\dfrac{y''(x)}{\big(1+G^{-1}(y(x))^2\big)^{\frac{5}{4}}} = \dfrac{\vb_*(\va)}{(1+\va^2)^{\frac{5}{2}}}
\end{align}
with the initial data \eqref{Seq:4.06}. 
For short we denote the right-hand side of \eqref{Sy_a} by
\begin{align}\notag
\vc(\va) := \dfrac{\vb_*(\va)}{(1+\va^2)^{\frac{5}{2}}}
= -2\left(\int_0^{\va}\dfrac{1}{\sqrt{\va-t}}\dfrac{\dt}{(1+t^2)^{\frac{5}{4}}} \right)^2. 
\end{align}

\begin{lemma} \label{Slem:4.2} 
Let $u(x;\va)$ be the solution of \eqref{Seq:1.1} with \eqref{Su_*} for $\va>0$. 
Then
\begin{align}\label{Seq:4.31}
u \Big(\frac{1}{2};\va\Big) =\frac{\displaystyle \int_0^{\va}\frac{s}{\sqrt{\va-s}}\frac{\ds}{(1+s^2)^{\frac{5}{4}}}}{\displaystyle 2\int_0^{\va}\frac{1}{\sqrt{\va-s}}\frac{\ds}{(1+s^2)^{\frac{5}{4}}}}.
\end{align}
\end{lemma}
\begin{proof}
Let $y(x;\va)$ be the function given by $y(x;\va)=G(u'(x;\va))$.
Then $y(\cdot;\va)$ satisfies \eqref{Sy_a} and  
 $y(\cdot;\va)$ is strictly decreasing in $(0,1/2)$ by \eqref{Seq:1103-3}.
Using this $y$, we have
\begin{align*}
u\Big(\frac{1}{2};\va\Big)=\int_0^{\frac{1}{2}}u'(x;\va)\dx=\int_0^{\frac{1}{2}} G^{-1}(y(x;\va))\dx,
\end{align*}
where we used $u(0;\va)=0$.
We infer from the change of variables $y(x;\va)=s$ and \eqref{Seq:1103-3} that
\begin{align*}
u\Big(\frac{1}{2};\va\Big)
&=\int_{G(\va)}^0 G^{-1}(s)\frac{1}{y'(x;\va)}\ds \\
&=\int_{G(\va)}^0 G^{-1}(s) \left(-\frac{(1+\va^2)^{\frac{5}{4}}}{|\vb_*(\va)|^{\frac{1}{2}}}\frac{1}{\sqrt{2\va-2G^{-1}(s)}}\right) \ds. 
\end{align*}
By the change of variables $G^{-1}(s)=x$, we have
\begin{align*}
u\Big(\frac{1}{2};\va\Big)&=\frac{(1+\va^2)^{\frac{5}{4}}}{|\vb_*(\va)|^{\frac{1}{2}}} \int_0^{\va}\frac{x}{\sqrt{2\va-2x}}\frac{1}{(1+x^2)^{\frac{5}{4}}} \dx 
=\frac{\displaystyle \int_0^{\va}\frac{x}{\sqrt{\va-x}}\frac{\dx}{(1+x^2)^{\frac{5}{4}}}}{\displaystyle 2\int_0^{\va}\frac{1}{\sqrt{\va-s}}\frac{\ds}{(1+s^2)^{\frac{5}{4}}}}, 
\end{align*}
which is the desired formula. 
Here we used
\begin{align}\label{S0726-4}
\frac{|\vb_*(\va)|^{\frac{1}{2}}}{(1+\va^2)^{\frac{5}{4}}} =  \sqrt{2}\int_0^{\va}\frac{1}{\sqrt{\va-s}}\frac{\ds}{(1+s^2)^{\frac{5}{4}}}, 
\end{align}
which follows from \eqref{Sprop:4.2}.
We complete the proof.
\end{proof}
Let us set 
\begin{align}\label{Seq:1104-3}
I(\va):=\int_0^{\va}\frac{\sqrt{\va}}{\sqrt{\va-s}}\frac{\ds}{(1+s^2)^{\frac{5}{4}}}, 
\quad 
J(\va):=\int_0^{\va}\frac{\sqrt{\va}}{\sqrt{\va-x}}\frac{x}{(1+x^2)^{\frac{5}{4}}}\dx.
\end{align}
Then thanks to Lemma~\ref{Slem:4.2}, we notice that 
\[ u\Big(\frac{1}{2};\va\Big)=\frac{J(\va)}{2I(\va)}. \]

\begin{lemma} \label{Slem:4.3} 
For $J(\va)$ given by \eqref{Seq:1104-3}, 
 it holds that
\[
J'(\va) > 0 \quad \text{for} \quad \va>0.
\]
\end{lemma}
\begin{proof}
By the change of variables we reduce $J(\va)$ into
\begin{align}\label{Seq:1208-1}
 J(\va)= \int_0^1 \frac{t}{\sqrt{1-t}}\frac{\va^2}{(1+\va^2t^2)^{\frac{5}{4}}}\dt. 
\end{align}
Let ${}_2F_1$ denote the Gaussian hypergeometric function (cf.\ Definition~\ref{Sdef:hypergeom}).
Note  that for each $\va>0$
\begin{align}
\int_0^1 \frac{t}{\sqrt{1-t}}\frac{\va^2}{(1+\va^2t^2)^{\frac{5}{4}}}\dt
&= \frac{2}{3}\frac{\va^2}{1+\va^2} {}_2F_1\big[1, \tfrac{3}{2}; \tfrac{7}{4}; \tfrac{\va^2}{1+\va^2} \big] 
\label{Seq:2205-1} \\
&= \frac{2}{3}\frac{\va^2}{1+\va^2} \sum_{n=0}^{\infty} \frac{(1)_n (\tfrac{3}{2})_n}{(\frac{7}{4})_n \, n!} \Big(\frac{\va^2}{1+\va^2} \Big)^n \notag
\end{align}
(see Proposition~\ref{Sprop:2205-1} for a rigorous derivation).
Here $(x)_n$ is the Pochhammer symbol, and $(x)_n >0$ holds for any $n\in\N\cup\{0\}$ and $x>0$.
Therefore, setting 
\[ j(\va):= \frac{\va^2}{1+\va^2} \quad \text{for} \quad \alpha>0,  \]
we obtain 
\begin{align}
\begin{split}
J'(\va)&= \frac{2}{3}j'(\va) \sum_{n=0}^{\infty} \frac{(1)_n (\tfrac{3}{2})_n}{(\frac{7}{4})_n \, n!} j(\va)^n \\
&\qquad + \frac{2}{3}j(\va) j'(\va)\sum_{n=1}^{\infty} \frac{(1)_n (\tfrac{3}{2})_n}{(\frac{7}{4})_n \, (n-1)!} j(\va)^{n-1} >0
\end{split}
\end{align}
with the help of the fact that $0<j(\alpha)<1$.
The proof is complete.
\end{proof}

Next, we show that $\va\mapsto u(\frac{1}{2};\va)$ is strictly increasing. 
To this end, let us set
\[ \xi(x,\va):=\frac{\pd y}{\pd\va}(x;\va), \]
where we regarded $y(x;\va)$ as a function on $[0,\frac{1}{2}]\times\R_{>0}$.
Then it follows from \eqref{Sy_a} that
\begin{align}\label{Seq:1103-1}
\xi''-\vc(\va)f'(y)\xi-\vc'(\va)f(y)=0, 
\end{align} 
where $f(t)=(1+G^{-1}(t)^2)^{\frac{5}{4}}$ and $'=\frac{\pd}{\pd x}$.
By the definition $f'(t)\geq0$ holds for $t\geq0$.
Moreover, $y(0;\va)=G(\va)$, $y'(0;\va)=0$, and $y(\frac{1}{2};\va)=0$ imply that
\begin{align*}
\xi(0,\va)=G'(\va), \quad \xi'(0,\va)=0, \quad \text{and}\quad \xi\Big(\frac{1}{2},\va\Big)=0.
\end{align*}

\begin{proposition} \label{Sprop:4.3} 
Let $u(x;\va)$ be the solution of \eqref{Seq:1.1} with \eqref{Su_*} for $\va>0$. 
Then
\begin{align}\label{Seq:4.34}
\frac{d }{d \va} \left[u\Big(\frac{1}{2};\va\Big)\right] > 0 \quad\text{for} \quad \va>0.
\end{align}
\end{proposition}
\begin{proof}
To begin with, recall that $y(\cdot;\va)$ is decreasing in $[0,\frac{1}{2}]$ for each $\va>0$, in partucular,
\begin{align}\label{Seq:4.24}
0<y(x;\va)<y(0;\va)=G(\va) \quad \text{and} \quad y'(x;\va)<0 \quad\text{for}\quad 0<x <\frac{1}{2}.
\end{align}
We infer from \eqref{Seq:1103-3} that
\begin{align}\label{Seq:1103-5}
y'(x;\va)^2 &=- \vc(\va)\big(2\va-2G^{-1}(y(x;\va))\big).
\end{align}
Differentiating \eqref{Seq:1103-5} with respect to $\va$, we have 
\begin{align}\label{Seq:1103-6}
2y'\cdot\xi'
=-2\vc'(\va)\bigg[ \va-G^{-1}(y) \bigg]
-2\vc(\va)\bigg[ 1-\Big( 1+G^{-1}(y)^2 \Big)^{\frac{5}{4}} \xi\bigg].
\end{align}
Set $\Gamma_{+}:=\Set{\va\in(0,\infty) | \vc'(\va)>0}$ and $\Gamma_{-}:=\Set{\va\in(0,\infty) | \vc'(\va)\leq0}$. 
We divide the proof into two cases.

\smallskip

\textbf{Case I}. {\sl We show \eqref{Seq:4.34} for $\va\in\Gamma_{-}$.\ } 
Fix $\va\in\Gamma_{-}$ arbitrarily.
Combining \eqref{Seq:1103-6} with \eqref{Seq:4.24}, $\vc(\va)<0$, and $\vc'(\va)\leq0$, we deduce that
\begin{align*}
\text{
if $c\in(0,\tfrac{1}{2})$ satisfies $\xi(c,\va)\leq0$, then $c$ also satisfies $\xi'(c,\va)<0$.
}
\end{align*}
Therefore if there is a point $c\in(0,\frac{1}{2})$ such that $\xi(c,\va)=0$, then $\xi(x,\va)<0$ holds for $x\in(c,\frac{1}{2}]$, 
which contradicts $\xi(\frac{1}{2},\va)=0$. 
Therefore we may assume that $\xi(x,\va)\ne0$ for $x\in[0,\frac{1}{2})$.
This together with $\xi(0,\va)=G'(\va)=(1+\va^2)^{\frac{5}{4}}>0$ implies that
\begin{align}\label{Seq:1103-7}
\xi(x,\va)\geq 0 \quad \text{for}\quad 0\leq x \leq \frac{1}{2}.
\end{align} 
On the other hand, it follows from $u(0;\va)=0$ that
\begin{align*}
u\Big(\frac{1}{2};\va\Big) 
= \int_0^{\frac{1}{2}}u'(x;\va) \dx = \int_0^{\frac{1}{2}}G^{-1}(y(x;\va)) \dx,
\end{align*}
which in combination with \eqref{Seq:1103-7} gives
\begin{align}\label{Seq:4.35}
\frac{d }{d \va} \left[u\Big(\frac{1}{2};\va\Big)\right]
=\int_0^{\frac{1}{2}}\big(1+G^{-1}(y)^2\big)^{\frac{5}{4}}\xi(x,\va) \dx \geq0.
\end{align}
The last equality does not hold due to $\xi(0,\va)>0$. 
Therefore we obtain \eqref{Seq:4.34} for $\va\in\Gamma_{-}$.

\smallskip

\textbf{Case II}. {\sl Show \eqref{Seq:4.34} for $\va\in\Gamma_{+}$.\ } 
Since $y(\frac{1}{2};\va)=0$ and $\xi(\frac{1}{2},\va)=0$ hold for all $\va>0$, substituting $x={1}/{2}$ into \eqref{Seq:1103-6}, we obtain 
\begin{align}\label{Seq:1103-8}
2y'\Big(\frac{1}{2};\va\Big)\cdot\xi'\Big(\frac{1}{2},\va\Big)
=-2\Big( \va\vc'(\va)+\vc(\va) \Big).
\end{align}
To continue, we distinguish two subcases:

\smallskip

(i) 
{\sl We consider $\va\in\Gamma_{+}$ satisfying $\va\vc'(\va)+\vc(\va) \leq0$. }
Then such $\va$ satisfies 
\begin{align*}
\xi'\Big(\frac{1}{2},\va\Big)\leq 0,
\end{align*}
where we used \eqref{Seq:1103-8} and $y'(\frac{1}{2};\va)<0$. 
Here, combining straightforward calculations with \eqref{Sy_a} and \eqref{Seq:1103-1}, we obtain 
\begin{align}\label{Seq:1104-1}
\Big(\xi'y'-\xi\vc(\va)f(y) \Big)' =\vc'(\va)f(y)y' \quad \text{in}\quad (0,\tfrac{1}{2})
\end{align}
for each $\va>0$.
Assume that 
\[\xi (x_0,\bar{\va}) <0 \quad \text{for some} \quad 0< x_0 <\frac{1}{2} \]
holds for some $\bar{\va}\in\Gamma_{+}$ with $\bar{\va}\vc'(\bar{\va})+\vc(\bar{\va}) \leq0$.
This together with $\xi(\frac{1}{2},\bar{\va})=0$ implies that there exists $x_1\in(0,\frac{1}{2})$ satisfying $\xi(x_1,\bar{\va})<0$ and
$\xi'(x_1,\bar{\va})=0$.
Then integrating \eqref{Seq:1104-1} on $(x_1,\frac{1}{2})$, we have 
\[
\xi'\Big(\frac{1}{2},\bar{\va}\Big) y'\Big(\frac{1}{2};\bar{\va}\Big) 
+\xi(x_1,\bar{\va})\vc(\bar{\va})f(y(x_1;\bar{\va}))  =\vc'(\bar{\va})\int_{x_1}^{\frac{1}{2}}f(y)y'\dx.
\]
However, the left-hand side takes a non-negative value while the right-hand side is negative,
which is impossible.
Hence it holds that 
\[ \xi(x,\va)\geq0\quad\text{for}\quad x\in\Big(0,\frac{1}{2}\Big)\]
and for $\va\in\Gamma_{+}$ with $\va\vc'(\va)+\vc(\va)\leq0$. 
Similar to \eqref{Seq:4.35}, we have \eqref{Seq:4.34}.

\smallskip

(ii)
{\sl We assume that $\va\in\Gamma_{+}$ satisfies $\va\vc'(\va)+\vc(\va) >0$. }
The assumption implies that $(\va\vc(\va))'>0$.
Let $I(\va)$ and $J(\va)$ be given by \eqref{Seq:1104-3}. 
Then by the fact that
\[
\va\vc(\va)= -2\left(\sqrt{\va}\int_0^{\va}\frac{1}{\sqrt{\va-t}}\frac{\!\dt}{(1+t^2)^{\frac{5}{4}}}\right)^2 
= -2I(\va)^2,
\]
$\va\vc'(\va)+\vc(\va) >0$ gives $I'(\va)<0$.
Since $u(\frac{1}{2};\va)=J(\va)/2I(\va)$,  we have  
\begin{align*}
\frac{d }{d \va} \left[u\Big(\frac{1}{2};\va\Big)\right] =
\frac{J'(\va)I(\va)-J(\va)I'(\va)}{2I(\va)^2}
>\frac{J'(\va)I(\va)}{2I(\va)^2},
\end{align*}
where in the last inequality we used $I'(\va)<0$ and $J(\va)>0$. 
Combining this with Lemma~\ref{Slem:4.3}, we obtain \eqref{Seq:4.34}.
We complete the proof. 
\end{proof}

\begin{proposition} \label{Sprop:4.4} 
Let $\va>0$ and $u(x;\va)$ be the solution of \eqref{Seq:1.1} with \eqref{Su_*}. 
Then
\begin{align}\label{Seq:4.26}
u\Big(\frac{1}{2};\va\Big) \to 
c_* \quad \text{as}\quad \va\to\infty,
\end{align}
where $c_*$ is the constant given by \eqref{Seq:c_*}.
\end{proposition}
\begin{proof}
Let $\va\gg1$ and $I(\va)$ be given by \eqref{Seq:1104-3}. 
Set
\begin{align*}
I(\va)&= \int_0^{\va^{{3}/{4}}}\frac{1}{\sqrt{1-s/\va}}\frac{1}{(1+s^2)^{\frac{5}{4}}}\ds 
+ \int_{\va^{{3}/{4}}}^{\va} \frac{1}{\sqrt{1-s/\va}}\frac{1}{(1+s^2)^{\frac{5}{4}}}\ds \\
& =:I_1(\va) + I_2(\va).
\end{align*}
Since $(1+s^2)^{\frac{5}{4}}$ is integrable on $(0,\infty)$, we obtain
\[ I_1(\va) =\int_0^{\infty}\chi_{(0,\va^{{3}/{4}})} \frac{1}{\sqrt{1-s/\va}}\frac{1}{(1+s^2)^{\frac{5}{4}}} \ds\to \int_0^{\infty}\frac{1}{(1+s^2)^{\frac{5}{4}}}\ds \quad\text{as}\quad \va\to\infty.
\]
On the other hand, we observe that 
\begin{align*} 
|I_2(\va)|&\leq \frac{1}{(1+\va^{\frac{3}{2}})^{\frac{5}{4}}}\int_{\va^{{3}/{4}}}^{\va} \frac{1}{\sqrt{1-s/\va}}\ds 
=\frac{1}{(1+\va^{\frac{3}{2}})^{\frac{5}{4}}}\cdot2\va\sqrt{\va-\va^{\frac{3}{4}}}
\to 0
\end{align*}
as $\va\to\infty$. 
Therefore we obtain 
\begin{align} \label{Seq:4.30}
I(\va) \to \int_0^{\infty}\frac{1}{(1+s^2)^{\frac{5}{4}}}\ds=\frac{c_0}{2} \quad \text{as}\quad \va\to\infty.
\end{align}
By the same argument we have 
\begin{align} \label{Seq:4.33}
J(\va) \to \int_0^{\infty}\frac{x}{(1+x^2)^{\frac{5}{4}}}\dx =2
\quad\text{as}\quad\va\to\infty. 
\end{align}
Combining \eqref{Seq:4.31} with \eqref{Seq:4.30} and \eqref{Seq:4.33}, we obtain \eqref{Seq:4.26}.
\end{proof}

In fact, we can obtain another characterization of the limit of $u(x;\va)$ (see Appendix~\ref{Ssubsec:4.3}).

\begin{remark}
Even if $u(x;\va)$ exists in $x\geq1$, $u$ never satisfies $u(1;\va)=0$ for any $\va>0$
since the solution of \eqref{Seq:1.1} with $u(0)=u''(0)=u(1)=u'(\frac{1}{2})=0$ is limited to $u\equiv0$.
Therefore $u(x;\va)$ is far different from the function obtained in \cite{DG_07}. 
For the comparison, see Figure~\ref{Sfig:3}:
one is $u(x;\frac{1}{2})$; the other is the solution of \eqref{Seq:1.1} with $u(0)=u(1)=0$ and $u'(0)=\frac{1}{2}$, and both of them satisfy $u'(0)=\frac{1}{2}$.
\end{remark}


\begin{figure}[h]
\centering
\includegraphics[width=4.2cm]{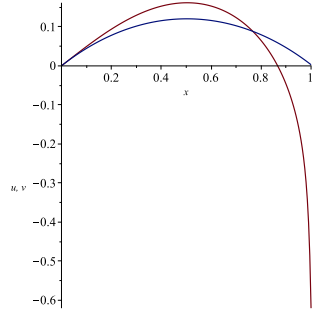}
\caption{$u(x;0.5)$ has a singularity.\label{Sfig:3}}
\end{figure}


\subsection{Proof of Theorem~\ref{Sthm:1.2}}\label{Ssubsec:4.2}

In this subsection, we turn to problem \eqref{Seq:P}. 
The results in Subsection~\ref{Ssubsec:4.2} also hold for \eqref{Seq:M} since the same argument is applicable.

\begin{proof}[Proof of Theorem~\ref{Sthm:1.2}]
\textsl{We first show uniqueness and existence.}
Assume that $\psi(\frac{1}{2}) < c_*$.
Then it follows from Proposition~\ref{Sprop:2.1} that \eqref{Seq:M} has a solution. 
As mentioned earlier, solutions of \eqref{Seq:M} also solve \eqref{Seq:P}.
Namely, a minimizer of $\W$ in $\Msym$ exists and it is a solution of \eqref{Seq:P}. 

We show the uniqueness of solutions of \eqref{Seq:P}.
By the argument in Section~\ref{Ssection:3}, every solution $u$ of \eqref{Seq:P} satisfies \eqref{SBVP} on $[0,\frac{1}{2}]$. 
For $H:=\psi(\frac{1}{2})\in(0,c_*)$, there exists $\tilde{\va}=\tilde{\va}(H)>0$ such that 
\[
u\Big(\frac{1}{2};\tilde{\va}\Big)=H.
\]
We infer from Proposition~\ref{Sprop:4.3} that such $\tilde{\va}$ is unique, so we obtain the conclusion.

\smallskip

\textsl{Next we show non-existence if $\psi(\tfrac{1}{2})\geq c_*$.}
Suppose that \eqref{Seq:P} has a solution $u$ under the assumption $\psi(\frac{1}{2}) \geq c_*$.
Then by the argument in Section~\ref{Ssection:3}, $u$ satisfies \eqref{SBVP} on $[0,\frac{1}{2}]$.
However, Proposition~\ref{Sprop:4.4} implies that  $u(\frac{1}{2};\va)$ cannot reach $c_*$ for any $\va:=u'(0)$, 
which contradicts our assumption. 
Therefore if $\psi(\frac{1}{2})\geq c_*$, \eqref{Seq:P} does not have a solution.

\smallskip

\textsl{We discuss the regularity.}
We assume that $\psi(\frac{1}{2}) < c_*$ and then \eqref{Seq:P} has a unique solution $u$.
Set $u'(0)=:\va>0$. 
Then $u$ satisfies $u|_{[0,1/2]}=u(\cdot;\va)|_{[0,1/2]}$, where $u(\cdot;\va)$ is the solution of \eqref{Seq:1.1} with \eqref{Su_*}.
We infer from \eqref{Seq:4.04} that
\[
\lim_{x \nearrow \frac{1}{2} } u'''(x;\va) = \frac{\vb_*(\va)}{(1+\va^2)^{\frac{5}{4}}}<0 
.
\]
However, if $u\in C^3(0,1)$, by symmetry $u'''(\tfrac{1}{2})=0$ must hold. 
This contradicts our assumption. 
\end{proof}

\begin{remark}
We refer to \cite{CF_79} as an example of the loss of regularity induced by obstacle. 
They considered a linear fourth order obstacle problem: 
\begin{align}\label{SYeq:1027-2}
\text{find}\quad u\in 
\mathcal{A}
\quad \text{such that}\quad \int_{\Omega}\Delta u\Delta(v-u)\dx\geq0 \quad \text{for} \quad v\in
\mathcal{A},
\end{align}
where  $\Omega\subset\R^N$ is a bounded domain and 
\[ 
\mathcal{A}
:=\Set{v\in H^2_0(\Omega) | v\geq\psi \ \ \text{a.e.\ in}\ \ \Omega }.\]
It is shown that there cannot exist an a priori $W^{3,p}(K)$ estimate on the solution of \eqref{SYeq:1027-2} for $p>n$ and for any compact subdomain $K\subset\subset\Omega$  (see \cite[Section 7]{CF_79}). 
\end{remark}


\section{Application to the parabolic problem}\label{Ssection:5}

Dynamical approaches are also useful to study variational problems. 
In this section we show that the solution of \eqref{Seq:M} obtained in Theorem~\ref{Sthm:1.1} can be characterized as  equilibrium of the corresponding parabolic problem: 
\begin{align} \label{Seq:GF} \tag{GF}
\begin{cases}
\partial_{t}u + \nabla\mathcal{W}(u) \geq 0           & \quad \text{in} \quad (0,1)\times(0,T),\\
\partial_{t}u + \nabla\mathcal{W}(u)  = 0 & \quad \text{in} \quad \{ (x,t)\in(0,1)\times(0,T)\ |\ u>\psi\},\\
u \geq \psi                                    & \quad \text{in} \quad (0,1)\times(0,T),\\ 
u = u '' = 0  & \quad \text{on} \quad \{0,1\}\times(0,T),\\
u(\cdot,0)=u_{0}(\cdot)                              & \quad \text{in} \quad (0,1), 
\end{cases}
\end{align}
where $\nabla\W(u)$ is the Euler--Lagrange operator of $\W(u)$, i.e.,
\begin{align*}
\int_0^1\nabla\W(u)\cdot\vp\dx&:=\frac{d}{d\ve} \W(u+\ve\vp)\big|_{\ve=0}\\
&\ =\int_0^1 \bigg[\Big(2\frac{u''}{(1+|u'|^2)^{\frac{5}{2}}}\Big)''+5\Big(\frac{|u''|^2u'}{(1+|u'|^2)^{\frac{7}{2}}}\Big)'\bigg]\vp\dx
\end{align*}
for $\vp\in C^{\infty}_{\rm c}(0,1)$.
If there is no obstacle, it is a standard matter to obtain global solvability and asymptotic stability for the $L^2$-gradient flow with initial data sufficiently close to stable equilibrium (see e.g.\ \cite{DG_09, NO_2017}).
However, since the solution of \eqref{Seq:M} and the solution of \eqref{Seq:GF} are not in general regular, 
the standard argument cannot work. 
In particular, \eqref{Seq:1.2} implies that
\textit{the solution of \eqref{Seq:GF} never converges to the solution of \eqref{Seq:M} in the $C^{\infty}$-topology}. 
A new ingredient for the following results is that 
\textit{we obtain global existence and asymptotic behavior in the full limit sense while the solution of \eqref{Seq:M} does not belong to $C^3(0,1)$}.

Let $u_0$ satisfy 
\begin{align}\label{Seq:5.1}
u_0\in H(0,1), \quad u_0(x)\geq \psi(x) \quad \text{for}\quad x\in [0,1].
\end{align}
For $T>0$, we define the convex set $\K_{T}$ by 
\begin{align*}
\K_T := \Set{ v \in L^{\infty}(0,T;H(0,1))\cap H^1(0,T;L^2(0,1)) | 
\begin{array}{l}
v \geq \psi \,\,\, \text{in} \,\,\, (0,1) \times [0,T], \\
v|_{t=0} = u_{0} \,\,\, \text{in}\,\,(0,1)
\end{array}
},
\end{align*}
where $H(0,1)$ is the Hilbert space $H(0,1)=H^2(0,1)\cap H^1_0(0,1)$ equipped with the scalar product
$$
(u,v)_{H(0,1)} := \int_0^1 u''v'' \, dx \quad \text{for}\quad u, v \in H(0,1). 
$$
In this paper we employ the norm $\|\cdot\|_{H(0,1)}$ on $H(0,1)$ as
$$
\|u\|_{H(0,1)} := (u,u)_{H(0,1)}^{1/2} \quad\text{for}\quad u \in H(0,1), 
$$
which is equivalent to $\|\cdot\|_{H^2(0,1)}$.
In fact, there exist $c_H,C_H>0$ such that 
\begin{align*}
c_H\|u\|_{H(0,1)}\leq \|u\|_{H^2(0,1)} \leq C_H\|u\|_{H(0,1)}
\end{align*}
(see e.g.\  \cite[Theorem 2.31]{GGS_2010}).

We formulate the definition of solutions to \eqref{Seq:GF} as follows.

\begin{definition}
We say that $u$ is a weak solution to \eqref{Seq:GF} in $(0,1) \times [0,T]$ if the following hold$\colon$  
\begin{enumerate}
\item[{\rm (i)}] $u \in \K_T;$  
\item[{\rm (ii)}]  For any $v\in\K_T$ it holds that
\begin{align} \label{Seq:1204-1}
\int_0^T\!\!\int_0^1\bigg[ \pd_{t} u (v-u) &+2\frac{u''(v-u)''}{(1+(u')^2)^{\frac{5}{2}}} -5 \frac{|u''|^2 u'(v-u)'} {(1+(u')^2)^{\frac{7}{2}}} \bigg]\dx\!\dt \geq 0.
\end{align}
\end{enumerate}
\end{definition}

We are now ready to state a local-in-time existence and uniqueness result proved in \cite{OY_2019}.

\begin{proposition}[Local-in-time existence and uniqueness, \cite{OY_2019}] \label{Sthm:5.1} 
Let $\psi$ satisfy 
\begin{align}\label{Seq:1216-1} 
\psi\in C([0,1]), \quad \psi(0)<0, \quad \psi(1)<0
\end{align}
and $u_0$ satisfy \eqref{Seq:5.1}. 
Then there exists a constant $T=T(u_0)>0$ 
such that \eqref{Seq:GF} has a unique weak solution in $(0,1)\times[0,T]$.
Moreover, $u(\cdot,t) \in H^2(0,1)$ for all $0\leq t \leq T$ and it holds that 
\begin{align}\label{Seq:5.02}
\W(u(t_2))-\W(u(t_1)) \leq -\frac{1}{2}\int_{t_1}^{t_2}\!\!\int_0^1 |\pd_t u|^2\dx\!\dt
\quad \text{for all} \quad 0\leq t_1 \leq t_2 \leq T.
\end{align}
\end{proposition}

\begin{remark}\label{Srem:5.1}
Let $u$ be the weak solution $u$ to \eqref{Seq:GF} in $(0,1)\times[0,T]$.
If there exists $M>0$ such that
\[
\sup_{0\leq t \leq T}\|u'(\cdot, t)\|_{L^{\infty}(0,1)} \leq M,  
\]
then $u$ can be uniquely extended to the solution in $(0,1)\times[0,T+L(u_0)]$, where
\begin{align*}
0< L(u_0):= c_H\left(1+M^2\right)^{\frac{5}{2}} \W(u_0)^{\frac{1}{2}}. 
\end{align*}
This extension is justified by solving \eqref{Seq:GF} with the initial datum $u(T)$
(more precisely, see \cite[proof of Theorem 1.1]{OY_2019}). 
Namely, it is important to deduce a uniform estimate for $u$ on $\dot{W}^{1,\infty}(0,1)$ to extend the time of existence of the solution.
\end{remark}


In order to discuss the asymptotic behavior of the solution $u$ of \eqref{Seq:GF}, we prepare two lemmas. 

\begin{lemma}[Preservation of symmetry] \label{Slem:5.1} 
Let $\psi$ be a symmetric function satisfying \eqref{Seq:1216-1} and $u$ be the solution of \eqref{Seq:GF} in $(0,1)\times[0,T]$ with the initial datum 
$
u_0 \in \Msym.
$ 
Then $u(\cdot, t)\in \Msym$ for $0\leq t \leq T$.  
\end{lemma}
\begin{proof}
Let $u$ be a weak solution of \eqref{Seq:GF} in $(0,1)\times[0,T]$ and $v\in \K_T$ be arbitrary.
Setting $\tilde{u}(x,t):=u(1-x,t)$ and $\tilde{v}(x,t):=v(1-x,t)$, we find that $\tilde{u}$, $\tilde{v}\in\K_T$ by symmetry of $\psi$ and $u_0$. 
Then 
\begin{align*}
\int_0^T\!\!\int_0^1\bigg[ \pd_{t} \tilde{u}& (v-\tilde{u}) +2\frac{\tilde{u}''(v-\tilde{u})''}{(1+(\tilde{u}')^2)^{\frac{5}{2}}} -5 \frac{|\tilde{u}''|^2 \tilde{u}'(v-\tilde{u})'} {(1+(\tilde{u}')^2)^{\frac{7}{2}}} \bigg]\dx\!\dt \\
&= \int_0^T\!\!\int_1^0-\bigg[ \pd_{t} u (\tilde{v}-u) +2\frac{u''(\tilde{v}-u)''}{(1+(u')^2)^{\frac{5}{2}}} -5 \frac{|u''|^2 u'(\tilde{v}-u)'} {(1+(u')^2)^{\frac{7}{2}}} \bigg]\dx\!\dt\geq0,
\end{align*}
where in the last inequality we used \eqref{Seq:1204-1} taking $\tilde{v}$ as the test function.
It follows from the uniqueness of solutions of \eqref{Seq:GF} that $\tilde{u}=u$
and this implies $u(\cdot,t)\in\Msym$ for any $0\leq t \leq T$.
\end{proof}

\begin{lemma} \label{Slem:5.2} 
Let $\psi$ satisfy \eqref{Seq:1216-1} and $u$ be the solution of \eqref{Seq:GF} in $(0,1)\times[0,T]$ with the initial datum $u_0 \in H(0,1)$ satisfying $\W(u_0)<c_0^2$. 
Then
\begin{align*}
\| u'(\cdot,t) \|_{L^{\infty}(0,1)}\leq G^{-1}\left( \frac{\sqrt{\W(u_0)}}{2} \right):=M^*(u_0) \quad \text{for all} \quad0\leq t\leq T.
\end{align*} 
\end{lemma}
\begin{proof}
By Lemma~\ref{Slem:5.1}, $u(\cdot,t)\in\Msym$ holds for $0\leq t \leq T$.
Fix $x\in[0,\frac{1}{2})$ and $t\in[0,T]$. 
Then it follows from the definition of $G$ that 
\begin{align} \label{Seq:2.03}
\W(u(t)) \geq \int_x^{1-x} u''(x,t)^2 G'(u'(x,t))^2\dx.
\end{align}
By the Cauchy--Schwarz inequality we have 
\begin{align*}
\begin{split}
\left| G(u'(x,t))-G(u'(1-x,t)) \right|
&=\left| \int_x^{1-x} G'(u'(x,t)) u''(x,t)  \dx \right| \\
&\leq \sqrt{1-2x}\left( \int_x^{1-x} G'(u'(x,t))^2  u''(x,t)^2  \dx \right)^{\frac{1}{2}} \\
&\leq \sqrt{1-2x}\ \W(u(t))^{\frac{1}{2}},
\end{split}
\end{align*}
where the last inequality holds by \eqref{Seq:2.03}.
By symmetry we have $u'(x,t)=-u'(1-x,t)$. 
Therefore since $G$ is an odd function, we obtain $G(u'(1-x,t))=-G(u'(x,t))$, and hence 
\begin{align*}
G(|u'(x,t)|)=|G(u'(x,t))|\leq \frac{1}{2}\W(u(t))^{\frac{1}{2}}.
\end{align*}
This together with \eqref{Seq:5.02} implies that
\begin{align*}
\|u'(x,t)\|_{L^{\infty}(0,1)}
\leq G^{-1}\left(\frac{\sqrt{\W(u(t))}}{2}\right)
\leq G^{-1}\left(\frac{\sqrt{\W(u_0)}}{2}\right),
\end{align*}
where we used the monotonicity of $G^{-1}$ and $\W(u_0)<c_0^2$. 
\end{proof}

We close this paper with the result on the relationship between \eqref{Seq:M} and \eqref{Seq:GF}:

\begin{theorem}[Asymptotic behavior] \label{Sthm:5.2} 
Let $\psi\in\SC$ and $\psi(\frac{1}{2})< c_*$, where $c_*$ is given by \eqref{Seq:c_*}. 
Let $u_0\in \Msym$ satisfy \eqref{Seq:5.1} and $\W(u_0)<c_0^2$. 
Then \eqref{Seq:GF} has a unique solution $u$ in $(0,1)\times[0,\infty)$. 
Moreover, $u$ satisfies 
\begin{align}\label{Seq:5.7}
u(\cdot,t) \to U \quad \text{in}\quad H^2(0,1) \quad \text{as} \quad t\to\infty,
\end{align}
where $U$ is the solution obtained by Theorem~\ref{Sthm:1.1}.
\end{theorem}

\begin{remark}\label{Srem:5.2}
As used in Lemma~\ref{Slem:5.2}, the assumption $\W(u_0)<c_0^2$ plays an important role for the uniform bound of $\|u'(\cdot,t)\|_{L^{\infty}(0,1)}$ (see \cite[Lemma 2.4]{DD} and \cite[Corollary 5.16]{Muller20}).
Therefore, no matter whether $\psi(\frac{1}{2})<c_*$ or not, we can obtain the global-in-time existence result under the assumption that $\W(u_0)<c_0^2$.
\end{remark}

\begin{proof}[Proof of Theorem~\ref{Sthm:5.2}]
%
%
\textsl{We first show the global-in-time existence.}
Let $u$ be the solution in $(0,1)\times[0,T]$. 
Set
\[
L^* := c_H\left(1+M^*(u_0)^2\right)^{\frac{5}{2}} \W(u_0)^{\frac{1}{2}}.
\]
Since Lemma~\ref{Slem:5.2} gives the uniform estimate of $\|u'(\cdot,t)\|_{L^{\infty}(0,1)}$, we can extend the solution $u$ to $(0,1)\times[0,T+L^*]$ by considering \eqref{Seq:GF} with the initial datum $u(T)$, as mentioned in Remark~\ref{Srem:5.1}.
Then by Lemma~\ref{Slem:5.2} it holds that
\begin{align*}
\| u'(\cdot,t)\|_{L^{\infty}(0,1)} \leq G^{-1}\left(\frac{\sqrt{\W(u(T))}}{2} \right) \quad \text{for all}\quad T\leq t\leq T+L^*.
\end{align*}
This together with \eqref{Seq:5.02} gives 
\begin{align*}
\| u'(\cdot,t)\|_{L^{\infty}(0,1)} \leq G^{-1}\left( \frac{\sqrt{\W(u_0)}}{2} \right)=M^*(u_0) \quad \text{for all}\quad 0\leq t\leq T+L^*. 
\end{align*}
By solving \eqref{Seq:GF} with the initial datum $u(T+L^*)$, and by following the same argument as above, we can extend the solution $u(x,t)$ to $t=T+2L^*$ and it follows from \eqref{Seq:5.02} that 
\[
\W(u(T+2L^*)) \leq \W(u(T+L^*)) \leq \W(u_0),
\]
which yields
\begin{align*}
\| u'(\cdot,t)\|_{L^{\infty}(0,1)} \leq M^*(u_0) \quad \text{for all}\quad T+L^*\leq t\leq T+2L^*. 
\end{align*}
Repeating this argument, we can extend the solution to an arbitrary time and find that $u$ satisfies
\begin{align}\label{Seq:5.8}
\| u'(\cdot,t)\|_{L^{\infty}(0,1)} \leq G^{-1}\left( \frac{\sqrt{\W(u_0)}}{2} \right) 
\quad \text{for} \quad 0\leq t<\infty.
\end{align}
Thus the solution $u$ of \eqref{Seq:GF} with the initial datum $u_0$ can be extended to $(0,1)\times[0,\infty)$ and satisfies 
\begin{align} \label{Seq:5.06}
\int_0^{\infty}\!\!\int_0^1\Bigl[ \pd_{t} u (v-u) &+2\frac{u''(v-u)''}{(1+(u')^2)^{\frac{5}{2}}} -5 \frac{|u''|^2 u'(v-u)'} {(1+(u')^2)^{\frac{7}{2}}} \Bigr]\dx\!\dt \ge 0
\end{align}
for $v\in \K_{\infty}$.
Moreover, since it holds by \eqref{Seq:5.02}, the energy monotonicity, that
\begin{align*}
\int_0^1 \frac{|u''(x,t)|^2}{(1+M^*(u_0)^2)^{\frac{5}{2}}}\dx \leq \W(u(t)) \leq \W(u_0) \quad \text{for} \quad 0\leq t<\infty,
\end{align*}
we obtain the $H^{2}$-uniform boundness of $u(\cdot,t)$, that is, 
\begin{align}\label{Seq:5.9}
\|u''(\cdot,t)\|^2_{L^2(0,1)} \leq (1+M^*(u_0)^2)^{\frac{5}{2}}\W(u_0) \quad \text{for} \quad 0\leq t<\infty.
\end{align}
Furthermore, \eqref{Seq:5.8} and \eqref{Seq:5.9} clearly imply that $\|u(\cdot,t)\|_{H^2(0,1)}$ is uniformly bounded.

\smallskip

\textsl{Next we prepare $\omega$-limit set.}
We will prove this with the help of the argument used in \cite[Section 3.1]{KSW}. 
Let $u$ be the solution in $(0,1)\times[0,\infty)$ with the initial datum $u_0$. 
To begin with, we define $\omega$-limit set by
\[
\omega(u_0):=\Set{w \in C^1([0,1]) | u(t_k) \to w \text{ in } C^1([0,1]) \text{ for some sequence } t_k \to \infty  }
\]
and we shall show that
\begin{align}\label{Seq:5.08}
\omega(u_0)=\{ U\},
\end{align}
where $U$ is the unique solution of \eqref{Seq:P}.
Since every $w\in \omega(u_0)$ satisfies 
\[
w\geq\psi \quad \text{in}\quad [0,1],
\]
in order to obtain \eqref{Seq:5.08} it is sufficient to show that $w$ satisfies \eqref{Seq:1.4}. 

Fix $w\in\omega(u_0)$ arbitrarily.
Then there exists $\{t_k\}$ such that $u(\cdot,t_k)\to w$ in $L^2(0,1)$ with $1<t_k \to \infty$. 
Define a sequence of functions $\{u_k\}$ by
\[
u_k(x,\vt):= u(x, t_k+\vt), \quad (x,\vt) \in (0,1)\times (-1,1). 
\]
Since \eqref{Seq:5.02} and Step~1 imply that $\pd_t u \in L^2((0,1)\times(0,\infty))$, it holds that
\begin{align}\label{Seq:5.07}
\int_{-1}^1 \int_0^1 |\pd_t u_k|^2 \dx\!\dvt
&= \int_{t_k-1}^{t_k+1}\!\int_0^1 |\pd_t u|^2 \dx\!\dt 
\to 0 \quad \text{as}\quad k\to\infty.
\end{align}
Therefore, 
by \eqref{Seq:5.07} and the Cauchy--Schwarz inequality we have
\begin{align*}
\int_{-1}^1\int_0^1 |u(x,t_k+\vt)&-u(x,t_k)|^2\dvt\!\dx \\
&=\int_0^1\int_{-1}^1 \left|\int_{t_k}^{t_k+\vt}\pd_t u(x,s)\ds\right|^2\dx\!\dvt \\
&\leq \int_0^1\int_{t_k-1}^{\infty} \left|\pd_t u(x,s) \right|^2\ds\!\dx
\ \to0 \quad \text{as}\quad k\to\infty,
\end{align*}
which yields
\begin{align}\label{Seq:5.10}
u_k \to w \quad \text{in}\quad L^2(-1,1;L^2(0,1)).
\end{align}
Moreover, by the same argument as in \cite[Lemma 4.10]{OY_2019}, $u_k(\cdot,s)\in H^3(0,1)$ for a.e. $s\in(-1,1)$ and 
\begin{align}\label{Seq:5.100}
\int_{-1}^1\|u'''_k(\cdot,s)\|_{L^2(0,1)}^2\ds \leq C \left( \int_{-1}^1 \|\pd_{t}u_k(\cdot,s)\|^2_{L^2(0,1)}\ds +1\right)
\end{align}
holds and hence $\{u_k\}$ is uniformly bounded in $L^2(-1,1;H^3(0,1))$.
By the Aubin--Lions--Simon compactness theorem, we find that $w \in H^2(0,1)$ and there exists a subsequence, which we still denote by $\{u_k\}$, such that
\begin{align}\label{Seq:5.11}
u_k \to w \quad \text{in}\quad L^2(-1,1;H^2(0,1)) \quad\text{as}\quad k\to\infty.
\end{align}

Next, we show that $w$ satisfies \eqref{Seq:1.4}.
Fix $V\in M$ arbitrarily and take a function $\vz\in C^{\infty}_{\rm c}(-1,1)$ satisfying $0\leq\vz\leq 1$ and $\vz\not\equiv0$.
Since $u(x,t)+\vz(t-t_k)\big( V(x)-u(x,t) \big) $ belongs to $\K_{\infty}$, using this as a test function in \eqref{Seq:5.06} we have 
\begin{align*}
\int_{-1}^1\int_0^1 \bigg[ \pd_t u_k (V-u_k)  +2\frac{u_k''(V-u_k)''}{(1+(u_k')^2)^{\frac{5}{2}}} -5 \frac{|u_k''|^2 u_k'(V-u_k)'} {(1+(u_k')^2)^{\frac{7}{2}}}   \bigg]\vz \dx\!\dvt\geq 0,
\end{align*}
where we used the change of variables  $\vt = t-t_k$.
Combining this with \eqref{Seq:5.07} and \eqref{Seq:5.11}, letting $k\to\infty$, we obtain
\begin{align*}
\int_{-1}^1 \vz \dvt\left(\int_0^1 \bigg[2\frac{w''(V-w)''}{(1+(w')^2)^{\frac{5}{2}}} -5 \frac{|w''|^2 w'(V-w)'} {(1+(w')^2)^{\frac{7}{2}}} \bigg] \dx\right) \geq 0.
\end{align*}
Hence dividing the above by $\int_{-1}^1 \vz\dvt$, we find that $w$ satisfies \eqref{Seq:1.4} and $w$ is a solution of \eqref{Seq:P}. 
We have already shown in Theorem~\ref{Sthm:1.2} that a solution of \eqref{Seq:P} is unique if $\psi\in\SC$ satisfies $\psi(\frac{1}{2})< c_*$. 
Therefore we obtain \eqref{Seq:5.08}.

\smallskip

\textsl{We show the the full-limit convergence.}
First we show that 
\begin{align}\label{Seq:5.12}
u(\cdot,t) \to U \quad \text{in} \quad C^1([0,1])
\end{align}
as $t\to\infty$.
Suppose that \eqref{Seq:5.12} does not hold. 
Then there exists $\epsilon>0$ and sequence $t_j\to\infty$ such that
\begin{align}\label{Seq:5.14}
\| u(\cdot,t_j) -U\|_{C^1([0,1])} \geq \epsilon. 
\end{align}
However, \eqref{Seq:5.9} implies that $\{u(\cdot,t_j)\}$ is bounded in $H^2(0,1)$ and hence by the Rellich--Kondrachov compactness theorem there exist $\{t_{j_k}\} \subset \{t_{j}\}$ and $\bar{u}\in H^2(0,1)$ such that 
\[
u(\cdot,t_{j_k}) \to \bar{u} \quad \text{in}\quad C^1([0,1])
\]
as $k\to\infty$.
Since \eqref{Seq:5.08} yields $\bar{u}=U$, this contradicts \eqref{Seq:5.14}. 
Thus we obtain \eqref{Seq:5.12}.

Next we prove that
\begin{align}\label{Seq:5.16}
\frac{u''(\cdot,t)}{(1+u'(\cdot,t)^2)^{\frac{5}{4}}} \wto \frac{U''}{(1+(U')^2)^{\frac{5}{4}}} \quad \text{weakly in} \quad L^2(0,1)
\end{align}
as $t\to\infty$. 
Using $G$, given by \eqref{Sdef-G}, we see that
\[G\big(u'(x,t)\big)'=\frac{u''(x,t)}{(1+u'(x,t)^2)^{\frac{5}{4}}}, 
\quad G\big(U'\big)'=\frac{U''}{(1+(U')^2)^{\frac{5}{4}}}.
\]
For any $\phi\in C^1_{\rm c}(0,1)$, it holds that
\begin{align*}
\int_0^1\bigg[\frac{u''(x,t)}{(1+u'(x,t)^2)^{\frac{5}{4}}} - \frac{U''(x)}{(1+U'(x)^2)^{\frac{5}{4}}} \bigg]\phi\dx
&=-\int_0^1\bigg[G\big(u'(\cdot,t)\big) - G\big(U'\big) \bigg]\phi'\dx\\
&\to 0 \quad \text{as}\quad t\to\infty,
\end{align*}
where we used the continuity of $G$ and \eqref{Seq:5.12}. 
Therefore we obtain \eqref{Seq:5.16}.

Finally, we show that 
\begin{align}\label{Seq:5.17}
\int_0^1\bigg(\frac{u''(\cdot,t)}{(1+u'(\cdot,t)^2)^{\frac{5}{4}}}\bigg) ^2\dx \to
\int_0^1\bigg( \frac{U''}{(1+(U')^2)^{\frac{5}{4}}} \bigg) ^2\dx
\end{align}
as $t\to\infty$. 
We can regard the left-hand side of \eqref{Seq:5.17} as $\W(u(t))$. 
Since \eqref{Seq:5.02} implies that $\W(u(t))$ is non-increasing with respect to $t$, there exists $A\in\R$ such that 
\begin{align} \label{Seq:5.18}
A=\inf_{t>0}\W(u(t))=\lim_{t\to\infty}\W(u(t)). 
\end{align}
Suppose that $\W(U)<A$.
Similar to \cite[proof of Lemma 6.3]{OY_2019} we can construct $\{t_j\}_{j\in\N}$ satisfying 
\[ u(\cdot,t_j) \to U \quad \text{in}\quad H^2(0,1). \] 
This together with \eqref{Seq:5.18} implies that 
\[ A=\lim_{t\to\infty}\W(u(t)) =\lim_{j\to\infty}\W(u(t_j)) =\W(U), \]
which contradicts $\W(U)<A$ and we obtain $\W(U)\geq A$.
Since $U$ is a unique minimizer of $\W$ in $\Msym$, $\W(U)> A$ does not occur.
Therefore $\W(U)=A$, which in combination with \eqref{Seq:5.18} gives 
\[ \lim_{t\to\infty}\W(u(t)) =\W(U). \]
This clearly asserts that \eqref{Seq:5.17} holds.
It follows from \eqref{Seq:5.16} and \eqref{Seq:5.17} that 
\[
\frac{u''(\cdot,t)}{(1+u'(\cdot,t)^2)^{\frac{5}{4}}} \to \frac{U''}{(1+(U')^2)^{\frac{5}{4}}} \quad \text{in} \quad L^2(0,1)
\]
as $t\to\infty$. 
This together with \eqref{Seq:5.12} implies that \eqref{Seq:5.7} holds.
\end{proof}

\appendix

\section{The Gaussian hypergeometric functions}\label{Ssect:hypergeom}
In this section we introduce the Gaussian hypergeometric function and give the proof of \eqref{Seq:2205-1}.
\begin{definition}[The Gaussian hypergeometric function] \label{Sdef:hypergeom}
For parameters $a,b,c \in \R$, the Gaussian hypergeometric function ${}_2F_1[a,b;c;x]$ is defined by
\begin{align}\label{eq:def-hypergeom}
{}_2F_1[a,b;c;x]:=\sum_{k=0}^\infty \frac{(a)_k (b)_k}{(c)_k\, k!}x^k, \quad x<1, 
\end{align}
where $(a)_k$ denotes the Pochhammer symbol:
\begin{align*}
(a)_k:=\begin{cases}
1 \quad & \quad  k=0, \\
a(a+1)\cdots(a+n-1) & \quad k\geq1.
\end{cases}
\end{align*}
\end{definition}
In general, the Gaussian hypergeometric function is defined with the parameters $a,b,c \in \C$ and $x\in\C$ (see e.g.\ \cite[Definition 2.1.5]{AAR99}).
\begin{remark}
The series in the right-hand side of \eqref{eq:def-hypergeom} converges for $|x|<1$. 
Elsewhere ${}_2F_1$ is understood as the analytic continuation to the complex plane from which a line joining $1$ to $\infty$ deleted. 
\end{remark}
\begin{proposition}[{Pfaff's transformation formula, \cite[Theorem 2.2.5]{AAR99}}] \label{prop:Pfaff}
For each $a,b,c>0$ and $x<-1$, it follows that
\[
{}_2F_1[a,b;c;x]=\frac{1}{(1-x)^a} {}_2F_1\big[a,b;c; \tfrac{x}{x-1} \big].
\]
\end{proposition}

We are now ready to prove \eqref{Seq:2205-1} and the proof is given as the following proposition. 
The key idea is to use Proposition~\ref{prop:Pfaff} and \cite[Lemma C.5]{Muller01}. 
\begin{proposition}\label{Sprop:2205-1}
Let $\va>0$.
Then
\[
\int_0^1 \frac{t}{\sqrt{1-t}}\frac{\va^2}{(1+\va^2t^2)^{\frac{5}{4}}}\dt
= \frac{2}{3}\frac{\va^2}{1+\va^2} {}_2F_1\big[1, \tfrac{3}{2}; \tfrac{7}{4}; \tfrac{\va^2}{1+\va^2} \big].
\]
\end{proposition}
\begin{proof}
We infer from the binomial theorem that 
\begin{align*}
\int_0^1 \frac{t}{\sqrt{1-t}}\frac{\va^2}{(1+\va^2t^2)^{\frac{5}{4}}}\dt
&= \va^2 \int_0^1 \frac{t}{\sqrt{1-t}} 
 \sum_{k=0}^\infty
  \begin{pmatrix}
      -\frac{5}{4}  \\
      k
    \end{pmatrix}
    (\va^2t^2)^k\dt \\
&= \va^2  \sum_{k=0}^\infty
  \begin{pmatrix}
      -\frac{5}{4}  \\
      k
    \end{pmatrix}
    (\va^2)^k \int_0^1 \frac{t^{2k+1}}{\sqrt{1-t}} \dt.
\end{align*}
From the fact that for any $m\in \N$
\[
\int_0^1 \frac{t^{m}}{\sqrt{1-t}} \dt = \prod_{l=1}^m\frac{2l}{2l+1}
\]
(see e.g.\ \cite[Lemma C.4]{Muller01}), it follows that
\begin{align*}
\int_0^1 \frac{t}{\sqrt{1-t}}\frac{\va^2}{(1+\va^2t^2)^{\frac{5}{4}}}\dt
&= \va^2  \sum_{k=0}^\infty
  \begin{pmatrix}
      -\frac{5}{4}  \\
      k
    \end{pmatrix}
    (\va^2)^k \prod_{l=1}^{2k+1}\frac{2l}{2l+1}.
\end{align*}
Moreover, noting that for any $a\in \R$ and $k\in \N$
\begin{align*}
(-a)_k &= (-a)(-a+1)\cdots(-a+k-1) \\
&= (-1)^k a(a-1)\cdots(a-k+1)
=   (-1)^k \begin{pmatrix}
      a \\
      k
    \end{pmatrix}
    k!,
\end{align*}
we obtain 
\begin{align*}
\int_0^1 \frac{t}{\sqrt{1-t}}\frac{\va^2}{(1+\va^2t^2)^{\frac{5}{4}}}\dt
&= \va^2  \sum_{k=0}^\infty \frac{\big(\frac{5}{4}\big)_k}{k!}
    (-\va^2)^k \prod_{l=1}^{2k+1}\frac{2l}{2l+1} \\
&=\va^2  \sum_{k=0}^\infty \Bigg[\frac{(-\va^2)^k}{k!} \left(\frac{5}{4} \cdot \frac{9}{4} \cdots \frac{4k+3}{4} \right) \\
     & \qquad\qquad\qquad \times \frac{2^{2k+1}\prod_{l=1}^{2k+1} l }{3 \big(\prod_{l=1}^{k} (4l+3) \big)\big(\prod_{l=1}^{k} (4l-3) \big)} \Bigg]\\
     &= \frac{2}{3}\va^2 \sum_{k=0}^\infty \frac{(-\va^2)^k}{k!}\frac{\prod_{l=1}^{2k+1} l }{\prod_{l=1}^{k} (4l-3) } .
\end{align*}
Since we infer from definition that 
\[
 \prod_{l=1}^{k} (4l-3) = 4^k \prod_{l=1}^{k} \left(l-1+\tfrac{7}{4}\right) = 4^k \big(\tfrac{7}{4}\big)_k, 
\]
it turns out that
\begin{align*}
\int_0^1 \frac{t}{\sqrt{1-t}}\frac{\va^2}{(1+\va^2t^2)^{\frac{5}{4}}}\dt
&= \frac{2}{3}\va^2 \sum_{k=0}^\infty \frac{(-\va^2)^k}{k!}\frac{\big(\prod_{l=1}^{k} 2l \big)\big(\prod_{l=1}^{k} (2l+1) \big)}{4^k \big(\tfrac{7}{4}\big)_k} \\
&= \frac{2}{3}\va^2 \sum_{k=0}^\infty (-\va^2)^k\frac{\prod_{l=1}^{k} (2l+1) }{2^k \big(\tfrac{7}{4}\big)_k},
\end{align*}
where in the last equality we used the fact that $\prod_{l=1}^{k} 2l =2^k k!$.
Furthermore, we have
\[
\prod_{l=1}^{k} (2l+1) = 2^k \prod_{l=1}^{k} \big(l -\tfrac{3}{2}+1 \big) = 2^k \big(\tfrac{3}{2}\big)_k, 
\]
which in combination with $(1)_k=k!$ yields
\begin{align*}
\int_0^1 \frac{t}{\sqrt{1-t}}\frac{\va^2}{(1+\va^2t^2)^{\frac{5}{4}}}\dt
&= \frac{2}{3}\va^2 \sum_{k=0}^\infty (-\va^2)^k\frac{ \big(\tfrac{3}{2}\big)_k}{\big(\tfrac{7}{4}\big)_k} \\
&= \frac{2}{3}\va^2 \sum_{k=0}^\infty (-\va^2)^k\frac{(1)_k \big(\tfrac{3}{2}\big)_k}{k!\big(\tfrac{7}{4}\big)_k}
= \frac{2}{3}\va^2 {}_2F_1\big[1, \tfrac{3}{2}; \tfrac{7}{4}; -\va^2 \big].
\end{align*}
Combining this with Proposition~\ref{prop:Pfaff}, we obtain the desired expression.
\end{proof}

\section{Convergence to the singular curve}\label{Ssubsec:4.3}

Let $u(\cdot;\va)$ be the solution of \eqref{Seq:1.1} with \eqref{Su_*} for $\va>0$. 
Then one may conjecture that
\begin{align}\label{S0727-1}
u(\cdot;\va) \to U_0 \quad \text{as}\quad \va\to\infty,
\end{align}
where $U_0$ is defined by 
\begin{align}\label{S0727-2}
U_0(x):=
\begin{cases}
\dfrac{2}{c_0\sqrt[4]{1+G^{-1}(\frac{c_0}{2}-c_0x)^2}} \quad &\text{if} \ \ 0<x\leq\frac{1}{2}, \\
\quad\quad\quad\quad \quad0 &\text{if} \ \ x=0.
\end{cases}
\end{align}
This function $U_0$ is obtained by Deckelnick and Grunau in \cite{DG_07}, as a limit of the solution of \eqref{Seq:1.1} with some Navier boundary conditions (see \cite[equation 25]{DG_07}).
However, it is not easy to show \eqref{S0727-1} by the previous argument, due to the gap between $u''(0;\va)=0$ and $\lim_{x\downarrow0}U_0''(x)=-\infty$.

It is known that $U_0$ is not smooth as a graph but smooth as a curve in $\R^2$.
Focusing on this property, we discuss the convergence of $u(\cdot;\va)$ \textit{as a planar curve} by applying the shooting method to 
\begin{align}\label{S0727-3}
\vk_{ss} +\frac{1}{2}\vk^3 = 0.
\end{align}
If $u$ satisfies \eqref{Seq:1.1}, then the curvature of $(x,u(x))$, up to reparametrization, satisfies \eqref{S0727-3}. 
For the initial value problem on \eqref{S0727-3}, we refer the result in \cite[Proposition 3.3]{Lin96}.
\begin{proposition}[\cite{Lin96}]\label{Sprop:0727-1} 
Given any real numbers $\vp_0$ and $\dot{\vp}_0$ the unique solution of the initial value problem, 
\begin{align*}
\vp_{ss}(s)+ \frac{1}{2}\vp(s)^3=0, \quad \vp(0)=\vp_0, \ \  \vp_s(0)=\dot{\vp}_0 
\end{align*}
is given by 
\begin{align}\label{S0728-1}
\vp(s)=\sqrt{2}a\, {\rm cn}(as+b) \quad \text{with} \quad 
 \sqrt{2}a\, {\rm cn}(b)=\vp_0, \quad   -\sqrt{2}a^2\, {\rm sn}(b){\rm dn}(b)=\dot{\vp}_0.
 \end{align}
\end{proposition}
Here ${\rm cn}$,  ${\rm sn}$  and  ${\rm dn}$ denote Jacobi's elliptic functions with modulus $1/\sqrt{2}$.
Set
\[ K(\tfrac{1}{\sqrt{2}}):= \int_0^1\frac{1}{\sqrt{(1-x^2)(1-\frac{1}{2}x^2)}}\dx, \]
which is called the elliptic integral of the first kind of modulus $1/\sqrt{2}$. 
Since
\[
{\rm cn}(K(\tfrac{1}{\sqrt{2}}))=0,  \ \ {\rm sn}(K(\tfrac{1}{\sqrt{2}}))=1  \ \ \text{and}\ \  {\rm dn}(K(\tfrac{1}{\sqrt{2}}))=\frac{1}{\sqrt{2}}
\] 
hold, $b$ in \eqref{S0728-1} coincides with $K(\tfrac{1}{\sqrt{2}})$ if we take $\vp_0=0$. 
For more details of elliptic functions we refer the reader to \cite{Byrd}.

\begin{definition}\label{S0727def}
Let $\gamma_U : [0,1/2] \to \R^2$ denote the curve $(x, U_0(x))$. 
Let $s$ and $\kappa_U$ denote the arclength parameter and the curvature 
of $\gamma_U$, respectively.
\end{definition}
\begin{lemma} \label{Slem:0727-2} 
\begin{align}
L_U&:=\int_0^{\frac{1}{2}} \sqrt{1+U_0'(x)^2}\dx
=\frac{1}{2c_0}\int_{\R}\frac{dt}{(1+t^2)^{\frac{3}{4}}},  \label{S0727-a}\\
\vk_U(s)&=c_0\,{\rm cn}\left( \tfrac{c_0}{\sqrt{2}}s+K(\tfrac{1}{\sqrt{2}}) \right). \label{S0727-b}
\end{align}
\end{lemma}
\begin{proof}
To begin with, for $U_0$ defined by \eqref{S0727-2} it holds that
\begin{align}
U'_0(x) &= G^{-1}\left( \frac{c_0}{2}-c_0 x \right), \notag \\
U''_0(x) &= -c_0 \left( 1+G^{-1}(\tfrac{c_0}{2}-c_0 x)^2 \right)^{\frac{5}{4}}=-c_0\left( 1+ U'_0(x)^2 \right)^{\frac{5}{4}},  \label{S0727-4}
\end{align}
for $x\in(0,{1}/{2}]$.
Then we infer from the change of variables $G^{-1}(\tfrac{c_0}{2}-c_0 x)=t$ that
\begin{align*}
L_U&=\int_0^{\frac{1}{2}}\sqrt{1+G^{-1}(\tfrac{c_0}{2}-c_0 x)^2 }\dx \\
&=\int_{\infty}^0\sqrt{1+t^2}  \frac{dt}{(-c_0)(1+t^2)^{\frac{5}{4}}}
=\frac{1}{2c_0}\int_{\R}\frac{dt}{(1+t^2)^{\frac{3}{4}}},
\end{align*}
which is the desired formula \eqref{S0727-a}. 

Next we show \eqref{S0727-b}.
Recalling the definition of $\vc_U$, by \eqref{S0727-4} we have
\begin{align}\label{S0727-5}
\vk_U(x)=\frac{U''_0(x)}{(1+U'_0(x)^2)^{\frac{3}{2}}} = -\frac{c_0}{(1+U'_0(x)^2)^{\frac{1}{4}}}.
\end{align}
Since $s(x)=\int_0^x \sqrt{1+U'_0(y)^2}\dy$, 
 it holds that
\begin{align}\label{S0727-6}
\frac{d}{ds} \vk_U(s)&= \frac{d\vk_U}{dx}\frac{dx}{ds}=-\frac{c_0^2}{2\sqrt{1+U'_0(x)^2}}U_0'(x), \\
\frac{d^2}{ds^2} \vk_U(s)&=-\frac{c_0^2}{2}\frac{U''_0(x)}{(1+U'_0(x)^2)^{\frac{5}{2}}}, \notag
\end{align}
which in combination with \eqref{S0727-5} gives
\[ \frac{d^2}{ds^2}\vk_U(s) + \frac{1}{2}\vk_U(s)^3=0. \]
Moreover, thanks to \eqref{S0727-6} we obtain
\begin{align}\label{S0727-7}
\lim_{s\to0}\frac{d}{ds} \vk_U(s) = -\frac{c_0^2}{2}, 
\end{align}
and hence we find that
\begin{align}
\vk_U\ \ \text{is the solution of \eqref{S0727-3} with}\ \ \vp(0)=0 \ \ \text{and} \ \ \vp'(0) = -\frac{c_0^2}{2}.
\end{align}
Therefore Proposition~\ref{Sprop:0727-1} yields \eqref{S0727-b}. 
\end{proof}

Let us turn to $u(x;\va)$. 
To begin with, we prepare the following lemma.
\begin{lemma} \label{Slem:0727-3} 
Let $u(x;\va)$ be the solution of \eqref{Seq:1.1} with \eqref{Su_*}.
Then it holds that
\begin{align}
&L_{\va}:=\int_0^{\frac{1}{2}} \sqrt{1+u'(x;\va)^2}\dx =\frac{1}{2I(\va)}\int_0^{\va}\frac{\sqrt{\va}}{\sqrt{\va-x}} \frac{dx}{(1+x^2)^{\frac{3}{4}}}, \label{S0727-9}\\
&\lim_{\va\to\infty}L_{\va}=L_U,  \label{S0727-10}
\end{align} 
where $I(\va)$ is defined by \eqref{Seq:1104-3}.
\end{lemma}
\begin{proof}
Similar to Lemma~\ref{Slem:4.2} we have 
\begin{align*}
L_{\va}&=\int_0^{\frac{1}{2}} \sqrt{1+u'(x;\va)^2}\dx \\
&=\int^0_{G(\va)} \sqrt{1+G^{-1}(s)^2}\left(-\frac{(1+\va^2)^{\frac{5}{4}}}{|\vb_*(\va)|^{\frac{1}{2}}}  \right)\frac{ds}{\sqrt{2\va-2G^{-1}(s)}} \\
&=\frac{(1+\va^2)^{\frac{5}{4}}}{|\vb_*(\va)|^{\frac{1}{2}}} \int_0^{\va} \sqrt{1+x^2}\frac{1}{\sqrt{2\va-2x}}\frac{dx}{(1+x^2)^{\frac{5}{4}}} \\
&=\frac{1}{2I(\va)}\int_0^{\va} \frac{\sqrt{\va}}{\sqrt{\va-x}}\frac{dx}{(1+x^2)^{\frac{3}{4}}},
\end{align*}
where we used \eqref{S0726-4} and \eqref{Seq:1104-3} in the last equality. 
Hence we obtain \eqref{S0727-9}.
By the same argument as in \eqref{Seq:4.30} it follows that 
\[ \lim_{\va\to\infty}\int_0^{\va}  \frac{\sqrt{\va}}{\sqrt{\va-x}}\frac{dx}{(1+x^2)^{\frac{3}{4}}} =\int_0^{\infty}\frac{dx}{(1+x^2)^{\frac{3}{4}}} .
\]
Combining this with \eqref{Seq:4.30} and \eqref{S0727-a}, we obtain \eqref{S0727-10}.
\end{proof}

For each $\va>0$, let $k_{\va}$ be the solution of \eqref{S0727-3} with 
\begin{align*}
\ \vp(0)=0 \quad \text{and} \quad \vp'(0) = -\frac{2\sqrt{1+\va^2}}{\va}I(\va)^2.
\end{align*}
Then Proposition~\ref{Sprop:0727-1} yields
\begin{align}\label{S0806-3}
k_{\va}(\vt)=\frac{2I(\va)\sqrt[4]{1+\va^2}}{\sqrt{\va}} {\rm cn}\left(\frac{\sqrt{2}I(\va)\sqrt[4]{1+\va^2}}{\sqrt{\va}} \vt + K(\tfrac{1}{\sqrt{2}}) \right).
\end{align}
Using this $k_{\va}$, we define $\vc_{\va}\in C^{\infty}([0,L_{\va}];\R^2)$ by 
\begin{align} \label{S0728-2}
\vc_{\va}(\vt):=Q_{\va}\int_0^{\vt} \Big(  \cos\theta_{\va}(t),  \sin\theta_{\va}(t)  \Big)\dt \quad \vt\in[0,L_{\va}], 
\end{align}
where $L_{\va}$ is given by \eqref{S0727-9} and 
\[
\theta_{\va}(t):= \int_0^t k_{\va}(z)dz,  \quad
Q_{\va}:= \frac{1}{\sqrt{1+\va^2}}\left(
    \begin{array}{cc}
      1  & -\va \\
      \va & 1 
    \end{array}
  \right).
\]
Then we notice that $\vt\in[0,L_{\va}]$ is the arclength parameter of $\vc_{\va}$, and that $k_{\va}$ stands for the curvature of $\vc_{\va}$.
Moreover, it follows that 
\begin{align}\label{S0806-1}
\vc_{\va}(0)=(0,0), \quad \vc'_{\va}(0)=\left(\frac{1}{\sqrt{1+\va^2}}, \frac{\va}{\sqrt{1+\va^2}}\right).
\end{align}

We shall show that the image of $\vc_{\va}$ is equal to that of the curve $(x,u(x;\va))$.
As mentioned before, since $u(x;\va)$ satisfies \eqref{Seq:1.1},
the curvature $\vk_{\va}$ of $(x,u(x;\va))$ satisfies \eqref{S0727-3} after reparametrization.
By \eqref{Su_*}, the initial conditions in terms of $\vk_{\va}$ are given by
\begin{align*}
\vk_{\va}(0)&=\frac{u''(0;\va)}{(1+u'(0;\va)^2)^{\frac{3}{2}}} = 0, \\
\dot{\vk}_{\va}(0)
&=\frac{1}{\sqrt{1+u'(x;\va)^2}}\frac{d}{dx}\left(\frac{u''(x;\va)}{(1+u'(x;\va)^2)^{\frac{3}{2}}} \right)\bigg|_{x=0} \\
&=\frac{1}{\sqrt{1+u'(0;\va)^2}} \left(\frac{u'''(0;\va)}{(1+u'(0;\va)^2)^{\frac{3}{2}}}- 5\frac{u''(0;\va)^2u'(0;\va)}{(1+u'(0;\va)^2)^{\frac{5}{2}}}\right)\\
&=\frac{\vb_*(\va)}{(1+\va^2)^2}=-\frac{2\sqrt{1+\va^2}}{\va}I(\va)^2,
\end{align*} 
where we used \eqref{S0726-4} and \eqref{Seq:1104-3} in the last equality and \ $\dot{}$\  denotes the derivative with respect to the arclength parameter of $(x,u(x;\va))$.
Thus we find that 
\begin{align*}
 \vk_{\va}\ \text{is the solution of \eqref{S0727-3} with}\ \vp(0)=0 \ \text{and} \ \vp'(0) = -\frac{2\sqrt{1+\va^2}}{\va}I(\va)^2,
\end{align*}
which implies that  $\vk_{\va}=k_{\va}$ in $[0,L_{\va}]$.
Moreover, combining \eqref{S0806-1} with $u(0;\va)=0$ and $u'(0;\va)=\va$, we infer from fundamental theorem of plane curves that the image of $\vc_{\va}$ is equal to that of the curve $(x,u(x;\va))$. 
Then we can show the following theorem, which implies $(x,u(x;\va))$ converges to $U_0(x)$ in a sense of planar curves.
\begin{theorem} \label{Sthm:0813}
Let $\vc_{\va}$, $\vc_{U}$ be the curve defined by \eqref{S0728-2}, Definition~\ref{S0727def}, respectively.
Then 
\begin{align*} 
 \vc_{\va}(\tfrac{L_{\va}}{L_U}s) \to \vc_{U}(s)  \quad \text{uniformly on} \ [0,L_{U}]  
\end{align*}
as  $\va\to\infty$.
\end{theorem}
\begin{proof}
It follows from $\vc_{U}(0)=(0,0)$ and $\vc_{U}'(0)=(0,1)$ that
\begin{align}\label{S0728-3}
\vc_{U}(s)=\left(
    \begin{array}{cc}
      0  & -1 \\
      1 & 0 
    \end{array}
  \right)\int_0^s \Big(  \cos\theta_{U}(t),  \sin\theta_{U}(t)  \Big)\dt \quad s\in[0,L_U], 
\end{align}
where $\theta_{U}(t):= \int_0^t\vk_{U}(z)dz$.
Since it follows from \eqref{Seq:4.30} that
\begin{align*}
k_{\va}(z)&= \frac{2I(\va)\sqrt[4]{1+\va^2}}{\sqrt{\va}} {\rm cn}\left(\frac{\sqrt{2}I(\va)\sqrt[4]{1+\va^2}}{\sqrt{\va}} z+ K(\tfrac{1}{\sqrt{2}}) \right) \\
&\to c_0 \,{\rm cn}\left(\tfrac{c_0}{\sqrt{2}} z+ K(\tfrac{1}{\sqrt{2}}) \right) =\vk_U(z) \quad \text{uniformly on}\quad [0,L_U]
\end{align*}
as $\va\to\infty$, we observe that 
\begin{align}\label{S0813-1}
\theta_{\va}(t) \to \theta_{U}(t)\quad \text{uniformly on}\quad [0,L_U].
\end{align}
Moreover, it follows from \eqref{S0728-2} that for each $s\in [0,L_{U}]$
\begin{align*}
\vc_{\va}(\tfrac{L_{\va}}{L_U}s)
&=Q_{\va}\int_0^{\tfrac{L_{\va}}{L_U}s} \Big(  \cos\theta_{\va}(t),  \sin\theta_{\va}(t)  \Big)\dt \\
&=Q_{\va}\int_0^{s} \Big(  \cos\theta_{\va}(t),  \sin\theta_{\va}(t)  \Big)\dt +Q_{\va}\int_s^{\tfrac{L_{\va}}{L_U}s} \Big(  \cos\theta_{\va}(t),  \sin\theta_{\va}(t)  \Big)\dt\\
&=: I_1 + I_2.
\end{align*}
It follows from \eqref{S0728-3} and \eqref{S0813-1} that 
\begin{equation*}
I_1 \to \gamma_U(s) \quad \text{as} \quad \text{uniformly on}\quad [0,L_U]
\end{equation*}
as $\va\to\infty$.
Furthermore, by Lemma~\ref{Slem:0727-3} we see that 
\begin{equation*}
I_2 \to (0,0) \quad \text{as} \quad \alpha \to \infty. 
\end{equation*}
Therefore Theorem~\ref{Sthm:0813} follows. 
\end{proof}

\bibliographystyle{siam}
\bibliography{ref_Yoshizawa_siam}

\end{document}